\documentclass[10pt]{siamonline250211}

\usepackage{amssymb,mathrsfs,mathtools,bbm,mathdots,algpseudocode,siunitx,nicematrix,tabularx,booktabs}
\usepackage{lipsum,bm,comment,xurl}
\usepackage{float,adjustbox,todonotes,subfig}
\usepackage{enumitem}
\setlist[enumerate]{leftmargin=.5in}
\setlist[itemize]{leftmargin=.5in}

\usepackage{tikz}
\usepackage{tikz-cd}
\usetikzlibrary{shapes.geometric,arrows}
\tikzset{commutative diagrams/.cd}
\usepackage{pgfplots}
\pgfplotsset{compat=1.18}

\let\S\relax
\allowdisplaybreaks

\newsiamremark{remark}{Remark}
\newsiamremark{notation}{Notation}
\newsiamremark{example}{Example}
\newsiamremark{examples}{Examples}

\newcommand{\RR}{\mathbb{R}}

\newcommand{\NN}{\mathbb{N}}

\newcommand{\KK}{\mathbb{K}}
\newcommand{\CC}{\mathbb{C}}
\newcommand{\PP}{\mathbb{P}}

\newcommand{\A}{\mathcal{A}}

\newcommand{\C}{\mathcal{C}}
\newcommand{\B}{\mathcal{B}}
\newcommand{\S}{\mathcal{S}}

\newcommand{\W}{\mathcal{W}}
\newcommand{\Z}{\mathcal{Z}}
\newcommand{\cl}[1]{\mathcal{#1}}

\newcommand{\supp}{\mathrm{supp}}

\newcommand{\Sym}{\mathrm{Sym}}

\newcommand{\rank}{\mathrm{rank}}

\newcommand{\be}{\begin{equation}}
\newcommand{\ee}{\end{equation}}

\renewcommand{\d}[1]{\ensuremath{\mathrm{d}\,\!{#1}}}
\renewcommand{\restriction}{\mathord{\upharpoonright}}

\headers{Symmetric tensor decomposition on rational varieties}{M. Bechere, S. Kuhlmann, and B. Mourrain}

\title{Symmetric tensor decomposition on rational varieties}
\author{Matteo Bechere\thanks{Fachbereich Mathematik und Statistik, Universität Konstanz, 78457 Konstanz, Germany (\email{matteo.bechere@uni-konstanz.de})}
\and Salma Kuhlmann\thanks{Fachbereich Mathematik und Statistik, Universität Konstanz, 78457 Konstanz, Germany (\email{salma.kuhlmann@uni-konstanz.de})}
\and Bernard Mourrain\thanks{Inria Centre at Université Côte d’Azur, Sophia Antipolis, France (\email{bernard.mourrain@inria.fr})}
}

\begin{document}
\small
\maketitle
\begin{abstract}
    We study the Waring decomposition of symmetric tensors with nodes on a rational variety. We provide an explicit characterization of the existence of such a decomposition under a technical surjectivity assumption, and introduce an efficient algorithm to decompose this novel class of structured symmetric tensors. The framework directly generalizes Hankel tensors (Qi 2015) to the multivariate setting. We analyse in detail the case of toric varieties and rational curves. Proving the existence of a quadrature formula of even strength $2N$ with at most $N+1$ nodes that avoids a prescribed finite set of points, we establish new sharp upper bounds on the minimal number of nodes for quadrature formulae on rational curves. Numerical experimentation demonstrates the computational advantage of this approach, compared to classical direct approaches.
\end{abstract}
\section{Introduction}

Tensors are a central tool in many branches of mathematics, such as in linear algebra, differential geometry, and in other fields, such as signal processing, electrical engineering, statistics and finance, thanks to their modelling power. Hence, efficiently manipulating and understanding tensors is crucial for many real-world problems involving multidimensional data. The purpose of tensor decomposition is to express a tensor as a sum of simpler tensors. This process unveils hidden information and additional properties which are often invisible in the starting form. Moreover, the minimal length of a tensor decomposition, also known as rank, is a measure of the complexity of the tensor.
Among tensors, symmetric tensors occupy a special place: they remain invariant under permutations of their arguments, reflecting inherent symmetries in the system they model. This property makes them especially important in models where order does not matter, such as stress tensors in continuum mechanics or covariance matrices in statistics. Furthermore, a symmetric tensor can be identified with a homogeneous polynomial. This perspective opens the door to interactions with algebraic geometry. In the setting of homogeneous polynomials, symmetric tensor decompositions are also known as Waring decompositions.

The primary motivation of our work is to analyze decompositions of tensors with \emph{nodes}  in a certain algebraic variety $\cl X$. 
The general Waring  decomposition of symmetric tensors of degree $d$ is a special case where $\cl X$ is either the projective space $\mathbb{P}^{n}$ or its Veronese embedding $\cl V_{d,n}$ in degree $d$.
For the decomposition of multilinear tensors, $\cl X$ is the Segre variety. For further details on 
this geometric point of view and on methods for computing such decompositions, we refer to \cite{landsberg2012tensors,math6120314,Ballico2012Decomposition,symmetrictensordecomp,BERNARDI201351,MR3032635}.
The problem of decomposing a form as a weighted sum of $k$-th powers of degree-$h$ forms is known as  \emph{$k$-th Waring decomposition problem} 
\cite{MR3906540,MR4792418} and it corresponds to the  case where $\cl X$ is the Veronese embedding in degree $k$ of the space of degree $h$ forms.

Remarkably, such tensor decompositions play a central role in the study of multiplication tensors in algebraic complexity theory. In particular, the tensor associated with matrix multiplication encodes the complexity of matrix products: its rank and its border rank are closely tied to the exponent of matrix multiplication \cite{Strassen1969,landsbergMultiplication}.

The general problem of symmetric tensor decomposition on a variety given by equations has been studied in \cite{nie2020symmetrictensordecompositionsvarieties}, where tensor decompositions on $\cl X$ are obtained using a refinement of the techniques used in \cite{symmetrictensordecomp, MR3627453,MR3032635}. The case of varieties given by linear projections of projective varieties from linear subspaces is studied in \cite{XRankBernardiBallico}.

In this paper, we consider projective varieties, given as the image of a rational map $\bm q$.
The case where $\bm q$ is the one associated to all binary monomials of a given degree is considered in \cite{qiHankelTensors}. This leads to the concept of \emph{Hankel tensors}, introduced in \cite{HankelTensorsOriginal} and further studied in \cite{qiHankelTensors,HankelTensorRank}. They have been used in fields such as geophysics \cite{Adamo2014HankelOMP,Trickett2013HankelTensor} and data analysis \cite{SIGNORETTO2011861}.

Specializing to positive decompositions, the problem becomes strictly connected to solving finite-dimensional truncated moment problems, as is illustrated in \cref{prop:waringIFFcoeffsmomentsequence}. Furthermore, the celebrated result by Richter \cite[Satz 4]{Richter1957} guarantees the representing measure, if it exists, to be finitely atomic. This observation further links the aforementioned problems to quadrature formulae. Computing quadrature formulae has the goal of approximating integrals of polynomials by a finite weighted sum of evaluations of said polynomials at a prescribed set of nodes. The connection between positive Waring decompositions and quadrature formulae is also documented in \cite[Chap. 7]{reznick-SOEP}. Given that well-behaved functions can be approximated by polynomials, the computation of quadrature formulae finds countless applications in numerical integration for convolution integrals appearing naturally in computational molecular biology and finance \cite{MR3152336,MR3073359,Caflisch1997ValuationOM,MR1458765}. Consequently, extensive work has been done on computation of quadrature formulae \cite{MR3490787,MR507557,MR3082517,MR1209428,MR119417,MR269119,MR94632}, in particular through moment theory as in \cite{MR3879968}, and on the estimation of the minimum number of nodes needed \cite{MR1489255,MR561295,MR1846517,MR656522,MR1258086,MR33206}. Further related works are \cite{MR1938799,MR2216081,MR2099788,MR3409310}, addressing truncated moment problems on varieties defined by one polynomial equation $p(x,y)=0$ with $\deg(p)\leq 2$ and, in \cite{Zalar12082024}, addressing the truncated moment problem on curves in the plane. More recently, in \cite{MR3748596}, the authors study quadrature formulae with nodes lying on algebraic curves in the plane, by means of optimization techniques.

Computing such a minimal Waring decomposition of a tensor is not a simple task. In fact, it is known to be an NP-hard problem \cite{MR3144915}. Another motivation of our work is to demonstrate that if the tensor has a structure, this decomposition problem can be addressed more efficiently. 
Algorithmic aspects of this decomposition problem 
have only been partially investigated in \cite{symmetrictensordecomp,gamertsfelder}.

In what follows, we show that if the tensor is $\bm q$-Symmetric (see \cref{def:q-symm}), its decomposition reduces to the decomposition of a higher-order tensor in smaller dimension, with the same rank, which is easier to handle by standard tensor decomposition methods (see \cref{SectionNumericalExperimentations}). 

\emph{Contributions:} In this paper, we study the decomposition properties of $\bm q$-Symmetric tensors, which form a subspace of the vector space of symmetric tensors. For a choice of forms $\bm q$, such that the substitution map $W_{\bm q}$ is surjective, we prove in \cref{QSymOnlyDec} that the $\bm q$-Symmetric tensors are exactly those having a Waring decomposition on the rational variety $V_{\bm q}$ parametrized by ${\bm q}$ (i.e. the projective points corresponding to the Waring decomposition lie in the image of the rational map $\bm q$), by proving the existence of a length-preserving bijective correspondence between $\bm q$-Symmetric decompositions of a $\bm q$-Symmetric tensor $p$ and Waring decompositions of an associated form $\psi_{\bm q}(p)$ (see \cref{QSymMainTheorem}). Given the length-preserving property of the correspondence, our method also results in an effective method to obtain Waring decompositions of high-rank forms, if they lie in the space of $\bm q$-Symmetric tensors, as is presented in \cref{SectionExampleDecomp}.
In \cref{DualCone}, we prove that, over the real numbers, the cone dual to the cone of positive $\bm q$-Symmetric tensors is the cone of forms which are non-negative on $V_{\bm q}$.
A special instance of the binary monomial $\bm q$-Symmetric tensors, the Hankel tensors, were studied in \cite{qiHankelTensors}. We generalize this theory to higher dimensions and general maps $\bm q$. In \cref{RankBound} we also improve the bound in \cite[Thm. 4.1]{qiHankelTensors} on the length of a minimal $\bm q$-Symmetric decomposition of a $\bm q$-Symmetric tensor. Furthermore, we  improve the proof  of \cite[Thm. 3.1]{qiHankelTensors} and answer \cite[Question 4.2]{qiHankelTensors} positively.

Additionally, we link the  positive $\bm q$-Symmetric tensor decomposition problem to the problem of computing quadrature formulae supported on algebraic curves. Under an additional technical condition, in \cref{ComparisonWithRienerTuratti}, we improve the bound of  \cite[Thm.1.1]{Riener2025QuadratureRW} on the number of nodes of quadrature formulae. As a tool to prove the aforementioned theorem, in \cref{thm:auxiliary}, we prove that it is possible to construct a quadrature formula of strength $2N$ with $N+1$ nodes for a non-degenerate truncated moment sequence (i.e. the Hankel matrix is invertible) on the real line which avoids a finite set of points. This theorem treats the problem of excluding prescribed points from the node set, a topic that, to the best of our knowledge, has not been investigated in the past and is complementary to the question of prescribing the quadrature formula to have a specific node (see e.g. \cite[Corollary 2.3]{blekherman2020generalized}), or multiple prescribed nodes (see \cite[Thm. 3.1]{nailwal2024gaussianquadraturesprescribednodes}).

Furthermore, we address the surjectivity assumption on $W_{\bm q}$ by characterizing it in terms of the Hilbert function of the vanishing ideal of $V_{\bm q}$ and also by stating sufficient conditions on the set $\bm q$ (see \cref{HilbertFuncCharacterization} and \cref{SurjectivityConditions}). We formalize our results in \cref{AlgQSym} to compute $\bm q$-Symmetric decompositions of any $\bm q$-Symmetric tensor $p$, provided one is able to find Waring decompositions of $\psi_{\bm q}(p)$. \cref{AlgQSym} is implemented in the Julia package \href{https://github.com/matteobechere/QSymDecomposition.jl}{QSymDecomposition.jl} \cite{QSymDecomposition} and in \cref{Sec4ExplicitDecomp} we demonstrate its effectiveness at decomposing tensors of higher rank.

\emph{Organization of the paper:} The paper is organized as follows. In \cref{Sec2Preliminaries} we introduce preliminary definitions on Waring decompositions and classical apolarity theory. In \cref{Sec3qSymForms} we define the subspace of $\bm q$-Symmetric tensors, study the surjectivity condition on $W_{\bm q}$ and prove the main results on the decomposition properties of $\bm q$-Symmetric tensors. Furthermore, we specialize to the even-degree case to state some results about the cone of positive $\bm q$-Symmetric tensors and its dual. Additionally, we examine in depth the special case when $\bm q$ is a vector of monomials and also a vector of binary forms, corresponding respectively to the problem of computing Waring decompositions on toric varieties and rational curves. Finally, we apply the theory of binary $\bm q$-Symmetric tensors to the problem of computing quadrature formulae and bounding the required number of nodes. In \cref{Sec4ExplicitDecomp} we formalize our decomposition results into \cref{AlgQSym} and provide examples of explicit decompositions of $\bm q$-Symmetric tensors.

\section{Preliminaries}\label{Sec2Preliminaries}

In the following, $\NN=\{1,2,3,\ldots\}$ is the set of positive integers, $\NN_0=\NN\cup\{0\}$, $\KK=\RR$ or $\KK=\CC$, and $\PP^n(\KK)$ denotes the $n$-dimensional \emph{projective space}. Vectors are intended as column vectors.

$\KK[X_1,\ldots,X_n]_{\leq k}$ is the vector space of polynomials in $n$ variables, of degree at most $k\in\NN_0$ and coefficients in $\KK$. Via homogenization, $\KK[X_1,\ldots,X_n]_{\leq k}$ is isomorphic to the vector space $\S^k(\KK^{n+1})$ of \emph{homogeneous polynomials} (also called \emph{forms}) of degree $k$ in the $n+1$ variables $\bm X=(X_0,\ldots,X_n)$. Symmetric tensors of order $k$ and dimension $n+1$ can be identified with forms of degree $k$ in $n+1$ variables.

For a $\KK$-vector space $V$, its dimension is denoted by $\dim_\KK(V)$. The \emph{graded algebra} of forms in $n+1$ variables is denoted by:
$$ \S(\KK^{n+1})=\bigoplus_{k=0}^\infty \S^k(\KK^{n+1}) $$
The \emph{algebraic dual space} $(\S(\KK^{n+1}))^*$ is the space of linear functionals on $\S(\KK^{n+1})$. For $F\in (\S(\KK^{n+1}))^*$ and $p\in \S(\KK^{n+1})$, we define $p\star F$ as the linear functional $q\mapsto F(pq)$.

We use bold symbols for tuples, such as $\bm v=(v_i)_i$. The set of \emph{exponents} in $n+1$ variables and total degree $k$ is:
\begin{equation*}
    M_{k,n+1}=\Bigg\{\bm\alpha\in\NN_0^{n+1}:|\bm\alpha|=\sum_{i=0}^{n}\alpha_i=k\Bigg\}
\end{equation*}
We make use of the multinomial notation and define $\bm X^{\bm\alpha}:=\prod_{i=0}^nX_i^{\alpha_i}$. Furthermore, when ordering monomials, or exponents, we use the lexicographic order. We use analogous notation for powers of tuples. In particular, $\S^k(\KK^{n+1})=\mathrm{span}(\bm X^{\bm \alpha})_{\bm\alpha\in M_{k,n+1}}$. The set $\{\bm X^{\bm\alpha}\mid\bm\alpha\in M_{k,n+1}\}$ is the \emph{standard monomial basis}, with $\dim_\KK(\S^k(\KK^{n+1}))=:D(k,n+1)=\binom{n+k}{k}$. 

For an exponent $\bm \alpha=(\alpha_0,\ldots,\alpha_n)\in M_{k,n+1}$, we set $\bm \alpha!=\alpha_0!\cdots \alpha_n!$ and the multinomial coefficient $\binom{k}{\bm\alpha}=\frac{k!}{\alpha_0!\cdots\alpha_n!}$. For tuples $\bm v,\bm w\in\KK^{n+1}$, their \emph{Euclidean inner product} is $\langle \bm v,\bm w\rangle = \sum_{i=0}^nv_iw_i$. By analogy, we define $\langle \bm v,\bm X\rangle=\sum_{i=0}^nv_iX_i$. 

For $\Xi=\{\bm\xi_1,\ldots,\bm\xi_r\}\subset\PP^n(\KK)$, let $I(\Xi)=\{p\in\S(\KK^{n+1}):p(\bm\xi_i)=0 \text{ for } i=1,\ldots,r\}$ be the \emph{ideal of polynomials vanishing on $\Xi$}. For an ideal $I\lhd\S(\KK^{n+1})$, $I_l = I\cap \S^l(\KK^{n+1})$ is its degree $l$ \emph{homogeneous component}. If $I$ is homogeneous, the \emph{Hilbert function} of $I$ is defined by $H_I(l)=\dim_\KK(\S^l(\KK^{n+1})/I_l)$. For a polynomial $p(\bm X)\in\KK[X_0,\ldots,X_n]$, $\Z(p)=\{\bm x\in\KK^{n+1}\mid p(\bm x)=0\}$ is the \emph{zero-set} of $p$ in $\KK$.

A form $p\in\S^k(\KK^{n+1})$ is called \emph{generic} if its coefficient vector lies in a Zariski open dense subset of the affine space $\KK^{D(k,n+1)}$.

Given a linear map $L:V\to W$, its \emph{transpose} $L^*:W^*\to V^*;\;f\mapsto f\circ L$ is also known as the \emph{pullback of $f$ by $L$}. If $(V,\langle\cdot,\cdot\rangle_V)$ and $(W,\langle\cdot,\cdot\rangle_W)$ are finite-dimensional inner product spaces, the \emph{adjoint} of $L$ is the unique map $L^\dagger:W\to V$ such that $\langle L(v),w\rangle_W=\langle v,L^\dagger(w)\rangle_V$. 

For a matrix $A$, $A^\top $ denotes its transpose. For a tuple $\bm c=(c_0,\ldots,c_{l})\in\KK^{l+1}$, $\operatorname{Hankel}(\bm c)$ is the $(\lfloor l/2\rfloor+1)\times (\lceil l/2\rceil+1)$ \emph{Hankel matrix} with entries $c_{i+j}$. 

A subset $C$ of a finite-dimensional inner product space over $\RR$ is a \emph{convex cone} if it is closed under addition and multiplication by non-negative scalars. $C$ is a \emph{closed convex cone} if it is closed in the induced topology. The \emph{dual cone} is defined as $C^*=\{v\in V\mid\langle v,w\rangle\geq0 \text{ for all } w\in C\}$. Finally, the \emph{$n$-sphere} $\mathbb{S}^n$ is the set $\{\bm v\in\RR^{n+1}\mid\langle\bm v,\bm v\rangle=1\}$.

Throughout the paper, measures are assumed to be positive Borel measures. Let $\mu$ be a measure and $\bm\alpha\in\NN_0^n$. If the integral $m_{\bm\alpha}=\int X_1^{\alpha_1}\cdots X_n^{\alpha_n}\d\mu$ is finite, $m_{\bm\alpha}$ is called the $\bm\alpha\textsuperscript{th}$ moment of $\mu$. If all moments exist, the sequence $(m_{\bm\alpha})_{\bm\alpha\in\NN_0^n}$ is called the \emph{moment sequence} of $\mu$. For $k\in\NN$, $(m_{\bm\alpha})_{\bm\alpha\in\NN_0^n,|\bm\alpha|\leq k}$ is the \emph{truncated moment sequence} of $\mu$, up to degree $k$. For a (truncated) moment sequence $(m_{\bm\alpha})$, the linear functional $L_{(m_{\bm\alpha})}$ on the (truncated) polynomial ring $\RR[\bm X]$ given by $\bm X^{\bm\alpha}\mapsto m_{\bm\alpha}$ is called the \emph{moment functional} of $(m_{\bm\alpha})$. Let $\bm m=(m_{\bm\alpha})_{\bm\alpha\in\NN_0^n}$ be a multisequence and let $L_{\bm m}$ be the associated moment functional. The \emph{moment matrix} of $\bm m$ is the infinite matrix
\begin{equation*}
    M(\bm m):=\big(L_{\bm m}(\bm X^{\bm\alpha+\bm\beta})\big)_{\bm\alpha,\bm\beta\in\NN_0^n}=\big(m_{\bm\alpha+\bm\beta}\big)_{\bm\alpha,\bm\beta\in\NN_0^n}.
\end{equation*}
For $d\in\NN_0$, the \emph{truncated moment matrix of order $d$} is $M_d(\bm m):=\big(m_{\bm\alpha+\bm\beta}\big)_{|\bm\alpha|,|\bm\beta|\leq d}$ and is a principal submatrix of $M(\bm m)$. Thus, $M_d(\bm m)$ depends only on the moments of degree at most $2d$. In the univariate case, $M_d(\bm m)=\operatorname{Hankel}(m_0,\ldots,m_{2d})$. For a homogeneous truncated multisequence $\bm m=(m_{\bm\alpha})_{\bm\alpha\in M_{2d,n+1}}$, its associated homogeneous moment functional is the linear map
\begin{equation*}
    L_{\bm m}:\S^{2d}(\RR^{n+1})\longrightarrow\RR,\qquad L_{\bm m}\Big(\sum_{\bm\alpha\in M_{2d,n+1}}a_{\bm\alpha}\bm X^{\bm\alpha}\Big)=\sum_{\bm\alpha\in M_{2d,n+1}}a_{\bm\alpha}m_{\bm\alpha}.
\end{equation*}
The \emph{homogeneous moment matrix of order $d$} associated with $\bm m$ is $M_d^{\mathrm h}(\bm m)=\left(m_{\bm\alpha+\bm\beta}\right)_{\bm\alpha,\bm\beta\in M_{d,n+1}}$.
The \emph{Dirac measure} centered at a point $\bm\xi$ is the probability measure attaining $1$ at any measurable set containing $\bm\xi$ and $0$ otherwise.

\subsection{Waring decompositions}
For a non-zero form $p\in \S^k(\KK^{n+1})$, we consider the following decomposition problem: Find a positive integer $r$, called the \emph{length}, $\lambda_1,\ldots,\lambda_r\in\KK\setminus\{0\}$ called \emph{weights}, $\bm\xi_1,\ldots,\bm\xi_r\in\KK^{n+1}$ called \emph{nodes} corresponding to distinct points in the projective space $\PP^n(\KK)$, such that
\begin{equation}\label{WaringDecomposition}
    p=\sum_{i=1}^r\lambda_i\langle \bm\xi_i,\bm X\rangle^k:=\sum_{i=1}^r\lambda_i\left(\xi_{i,0} X_0 + \cdots + \xi_{i,n} X_n\right)^{k}.
\end{equation}
If $\bm\xi_1,\ldots,\bm\xi_r\in A$ for some subset $A\subseteq \KK^{n+1}$, $p$ \emph{has a Waring decomposition on} $A$.
Decompositions as in \cref{WaringDecomposition} always exist over a field $\KK$ of characteristic $0$ (for $\KK=\CC$ see e.g. \cite[Lemma 4.2]{comon_mourrain})\footnote{A modification of the proof of \cite[Lemma 4.2]{comon_mourrain} shows that Waring decompositions always exist also for $\KK=\RR$.} and are known as \emph{Waring decompositions} over $\KK$. The smallest $r\in\NN$ for which a Waring decomposition is achieved, is known as the \emph{Waring rank} of $p$ over $\KK$. Note that when $\KK=\CC$, then the Waring rank is the smallest possible.

\subsection{Apolarity}
Denote the dual basis of the monomial basis $\{\bm X^{\bm\alpha}\}_{\alpha\in M_{k,n+1}}\subset\S^k(\KK^{n+1})$ by $\{\bm Y^{\bm \alpha}\}_{\bm\alpha\in M_{k,n+1}}\subset(\S^k(\KK^{n+1}))^*$.
\begin{definition}
    For $p=\sum_{\bm\alpha\in M_{k,n+1}}p_{\bm\alpha}\bm X^{\bm\alpha},q=\sum_{\bm\alpha\in M_{k,n+1}}q_{\bm\alpha}\bm X^{\bm\alpha}\in\S^k(\KK^{n+1})$ we define the \emph{apolar product} of $p$ and $q$ as
    \begin{equation*}
        \langle p,q\rangle_k:=\sum_{\bm\alpha\in M_{k,n+1}}\binom{k}{\bm\alpha}^{-1}p_{\bm\alpha}q_{\bm\alpha}
    \end{equation*}
\end{definition}
Moreover, the \emph{apolarity map} $A_k$ is defined as the $\KK$-linear isomorphism
\begin{equation}\label{ApolarityMap}
        \begin{aligned}
            A_{k}:\S^{k}(\KK^{n+1})&\longrightarrow (\S^{k}(\KK^{n+1}))^*\\
            p&\longmapsto \langle p,\cdot\rangle_k
        \end{aligned}
\end{equation}
Note that for all $\bm\alpha\in M_{k,n+1}$, $A_k(\bm X^{\bm\alpha})= \binom{k}{\bm\alpha}^{-1}(\bm Y^{\bm\alpha})$.
\begin{definition}[Catalecticant maps]\label{CatalecticantMaps}
    For $0\leq l\leq k$ and $p\in\S^k(\KK^{n+1})$ we let $C_p^l$ be the $\KK$-linear map defined as
    \begin{equation*}
        \begin{aligned}
            C_p^{l}:\S^{k-l}(\KK^{n+1})&\longrightarrow (\S^{l}(\KK^{n+1}))^*\\
            f&\longmapsto f\star A_k(p)
        \end{aligned}
    \end{equation*}
    The \emph{$l\textsuperscript{th} $catalecticant matrix of $p$} is the matrix associated to $C_p^{l}$ with respect to the bases \linebreak $\{\bm X^{\bm \alpha}\}_{\bm\alpha\in M_{k-l,n+1}}$ of $\S^{k-l}(\KK^{n+1})$ and $\{\bm Y^{\bm \beta}\}_{\bm\beta\in M_{l,n+1}}$ of $(\S^{l}(\KK^{n+1}))^*$.
\end{definition}
We also denote the catalecticant matrix by $C_p^{l}$; it will be clear from the context when we refer to the linear map or the matrix. For $p\in\S^k(\KK^{n+1})$, the \emph{apolar ideal} or \emph{annihilator} of $p$ is $p^\perp=\{f\in\S(\KK^{n+1}):f\star A_k(p)=0\}$ and is, in fact, an ideal of $\S(\KK^{n+1})$.
\begin{remark}\label{rem:catalecticantproperties}
We observe the following properties of the catalecticant maps:
\begin{enumerate}
    \item Note that $C_p^l$ is a square matrix if and only if $k$ is even and $l=\frac{k}{2}$. By convention, in general we shall call the \emph{most nearly square catalecticant matrix} the one obtained for $l=\lfloor \frac{k}{2}\rfloor$. Moreover, for all $0\leq l\leq k$, the transpose matrix of $C_p^{l}$ is $C_p^{k-l}$.
    \item
    By \cref{CatalecticantMaps} it follows that $\ker(C_p^{l})$ consists of degree $k-l$ forms annihilating $p$, i.e. $\ker(C_p^{l})=(p^\perp)_{k-l}$, thus $p^\perp$ can be studied through the catalecticant maps.
\end{enumerate}
    
\end{remark}
The main ingredient to study Waring decompositions is
\begin{lemma}[Apolarity Lemma]\label{ApolarityLemma}
        Let $p\in\S^k(\CC^{n+1})$. Then $p=\sum_{i=1}^r\lambda_i\langle \bm\xi_i,\bm X\rangle^k$ if and only if $I(\{\bm\xi_1,\ldots,\bm\xi_r\})\subseteq p^\perp$, for $r\in\NN,\,\lambda_1,\ldots,\lambda_r\in\CC\setminus\{0\},\,\bm\xi_1,\ldots,\bm\xi_r\in\CC^{n+1}$ corresponding to different points in $\PP^n(\CC)$.
\end{lemma}
A proof of the \cref{ApolarityLemma} can be found in \cite[Lemma 1.15]{IK}.

\section{\texorpdfstring{$\bm q$}{q}-Symmetric tensors}\label{Sec3qSymForms}
Throughout this section, let $m,n,k,h$ be positive integers, $s=D(k,\allowbreak n+1)$ and $t=D(hk,m+1)$.

\begin{definition}[Weight function]\label{DefWeightFunction}
    Let ${\bm q}=(q_0(\bm Z),\ldots,q_n(\bm Z))$,  with $q_i\in\S^h(\KK^{m+1})$, be a tuple of forms. The linear map
    \begin{equation}\label{WeightFunction}
        \begin{aligned}
            W_{\bm q}:\S^k(\KK^{n+1})&\longrightarrow \S^{hk}(\KK^{m+1})\\
            \bm X^{\bm \alpha}&\longmapsto W_{\bm q}(\bm X^{\bm \alpha}):=\bm q(\bm Z)^{\bm\alpha}=\prod_{i=0}^n q_i(\bm Z)^{\alpha_i}
        \end{aligned}
    \end{equation}
    for $\bm\alpha\in M_{k,n+1}$, is called \emph{weight function}. We denote by $\mathrm{Im}(W_{\bm q})=\mathrm{span}(W_{{\bm q}}(\bm X^{\bm\alpha})\mid\bm\alpha\in M_{k,n+1})=\mathrm{span}(\bm q^{\bm\alpha}\mid\bm\alpha\in M_{k,n+1})$ the image of $W_{\bm q}$.
Let $\W=(W_{\bm\gamma_i,\bm\alpha_j})_{1\leq i\leq t,1\leq j\leq s}$ be the matrix associated to $W_{\bm q}$ with respect to the monomial bases, i.e. $W_{\bm q}(\bm X^{\bm\alpha})=\sum_{\bm\gamma\in M_{hk,m+1}}\!\!W_{\bm\gamma,\bm\alpha} \bm Z^{\bm\gamma}$. Furthermore, we consider the following diagram
\begin{equation}\label{Phi}
\begin{tikzcd}[column sep=large, row sep=small]
\S^{hk}(\KK^{m+1})
  \arrow[r, "A_{hk}"]
& (\S^{hk}(\KK^{m+1}))^*
  \arrow[r, "W_{\bm q}^*"]
& (\S^{k}(\KK^{n+1}))^*
  \arrow[r, "A_k^{-1}"]
& \S^{k}(\KK^{n+1})
\end{tikzcd}
\end{equation}
and define $\phi_{\bm q}=A_{k}^{-1}\circ W_{\bm q}^*\circ A_{hk}$.
\end{definition}

\begin{proposition}\label{PhiAdjointW}
    The map $\phi_{\bm q}$ coincides with the adjoint $W_{\bm q}^\dagger$ of $W_{\bm q}$ with respect to the apolar products $\langle\cdot,\cdot\rangle_{hk}$ on $\S^{hk}(\KK^{m+1})$ and $\langle\cdot,\cdot\rangle_{k}$ on $\S^{k}(\KK^{n+1})$, i.e. for all $p\in\S^k(\KK^{n+1}),\,p'\in\S^{hk}(\KK^{m+1})$ the equality
    \begin{equation*}
        \langle W_{\bm q}(p),p'\rangle_{hk}=\langle p,\phi_{\bm q}(p')\rangle_k
    \end{equation*}
    holds.
\end{proposition}
\begin{proof}
    We first compute the explicit expression for the adjoint $W_{\bm q}^\dagger$ of $W_{\bm q}$.\\
    Let $p=\sum_{\bm\alpha\in M_{k,n+1}}p_{\bm\alpha}\bm X^{\bm\alpha}\in\S^k(\KK^{n+1})$ and $p'=\sum_{\bm\gamma\in M_{hk,m+1}}p'_{\bm\gamma}\bm Z^{\bm\gamma}\in\S^{hk}(\KK^{m+1})$. Substituting
    \begin{equation*}
        W_{\bm q}(p)=\sum_{\bm\gamma\in M_{hk,m+1}}\bm Z^{\bm\gamma}\sum_{\bm\alpha\in M_{k,n+1}}p_{\bm\alpha}W_{\bm\gamma,\bm\alpha},
    \end{equation*}
    into the inner product $\langle W_{\bm q}(p),p'\rangle_{hk}$ we obtain
    \begin{align*}
        \langle W_{\bm q}(p),p'\rangle_{hk}&=\sum_{\bm\gamma\in M_{hk,m+1}}\binom{hk}{\bm\gamma}^{-1}p'_{\bm\gamma}\sum_{\bm\alpha\in M_{k,n+1}}p_{\bm\alpha}W_{\bm\gamma,\bm\alpha}
        =\sum_{\bm\alpha\in M_{k,n+1}}p_{\bm\alpha}\sum_{\bm\gamma\in M_{hk,m+1}}\binom{hk}{\bm\gamma}^{-1}p'_{\bm\gamma}W_{\bm\gamma,\bm\alpha}\\
        &=\sum_{\bm\alpha\in M_{k,n+1}}\binom{k}{\bm\alpha}^{-1}p_{\bm\alpha}\left(\binom{k}{\bm\alpha}\sum_{\bm\gamma\in M_{hk,m+1}}\binom{hk}{\bm\gamma}^{-1}p'_{\bm\gamma}W_{\bm\gamma,\bm\alpha}\right)\\
        &=\left\langle p,\sum_{\bm\alpha\in M_{k,n+1}}\binom{k}{\bm\alpha}\sum_{\bm\gamma\in M_{hk,m+1}}\binom{hk}{\bm\gamma}^{-1}p'_{\bm\gamma}W_{\bm\gamma,\bm\alpha}\bm X^{\bm\alpha}\right\rangle_k
    \end{align*}
    Hence, we deduce that
    \begin{equation}\label{AdjointOfW}
        W_{\bm q}^\dagger(p')=\sum_{\bm\alpha\in M_{k,n+1}}\binom{k}{\bm\alpha}\sum_{\bm\gamma\in M_{hk,m+1}}\binom{hk}{\bm\gamma}^{-1}p'_{\bm\gamma}W_{\bm\gamma,\bm\alpha}\bm X^{\bm\alpha}.
    \end{equation}
    On the other hand, recall that $A_{hk}(p')=\langle p',\cdot\rangle_{hk}$. Therefore, $W_{\bm q}^*(A_{hk}(p'))=\langle p',W_{\bm q}(\cdot)\rangle_{hk}$ is the linear functional defined by
    \begin{equation*}
        p=\sum_{\bm\alpha\in M_{k,n+1}}p_{\bm\alpha}\bm X^{\bm\alpha}\mapsto\langle p',W_{\bm q}(p)\rangle_{hk}=\sum_{\bm\gamma\in M_{hk,m+1}}\binom{hk}{\bm\gamma}^{-1}p'_{\bm\gamma}\sum_{\bm\alpha\in M_{k,n+1}}p_{\bm\alpha} W_{\bm\gamma,\bm\alpha}.
    \end{equation*}
    Applying $A_k^{-1}$ yields
    \begin{equation}\label{ExplicitPsi}
        \phi_{\bm q}(p')=A_k^{-1}(W_{\bm q}^*(A_{hk}(p')))=\sum_{\bm\alpha\in M_{k,n+1}}\binom{k}{\bm\alpha}\sum_{\bm\gamma\in M_{hk,m+1}}\binom{hk}{\bm\gamma}^{-1}p'_{\bm\gamma}W_{\bm\gamma,\bm\alpha}\bm X^{\bm\alpha}
    \end{equation}
    Comparing \cref{ExplicitPsi,AdjointOfW}, we conclude that $\phi_{\bm q}=W_{\bm q}^\dagger$.
\end{proof}
As a direct consequence of \cref{PhiAdjointW}, we obtain the following corollary.
\begin{corollary}\label{WSurjPhiInj}
The map $W_{\bm q}$ is surjective if and only if the map $\phi_{\bm q}$ is injective.
\end{corollary}
\begin{proof}
    The statement follows from \cref{PhiAdjointW} and the properties of adjoint linear operators on finite-dimensional inner product spaces.
\end{proof}

\begin{definition}[${\bm q}$-Symmetric tensors]\label{def:q-symm}
Let ${\bm q}=(q_0(\bm Z),\ldots,q_n(\bm Z))$,  with $q_i\in\S^h(\KK^{m+1})$, be a tuple of forms and let $\phi_{\bm q}$ be as in \cref{Phi}. The subspace of $\S^k(\KK^{n+1})$ of \emph{${\bm q}$-Symmetric tensors} of degree $k$ in the $n+1$ variables $\bm X=(X_0,\ldots,X_n)$ is defined as ${\bm q}$-$\Sym_{k,n+1}:=\mathrm{Im}(\phi_{\bm q})$.
\end{definition}

\begin{remark}
    Since, by \cref{PhiAdjointW}, $\phi_{\bm q}=W_{\bm q}^\dagger$, we compute
    \begin{equation*}
        \dim({\bm q}\text{-}\Sym_{k,n+1})=\dim(\mathrm{Im}(\phi_{\bm q}))=\dim(\mathrm{Im}(W_{\bm q}^\dagger))=\dim(\mathrm{Im}(W_{\bm q})).
    \end{equation*}
\end{remark}

\begin{proposition}
    If $W_{\bm q}$ is surjective, then $\S^{hk}(\KK^{m+1})\xrightarrow{\phi_{\bm q}}{\bm q}$-$\Sym_{k,n+1}$ is an isomorphism. Moreover, a basis of ${\bm q}$-$\Sym_{k,n+1}$ is given by
    \begin{equation*}
        \B:=\left\{b_{\bm\gamma}:=\phi_{\bm q}(\bm Z^{\bm\gamma})\in\S^k(\KK^{n+1})\mid\bm\gamma\in M_{hk,m+1}\right\}
    \end{equation*}
    where each basis element $b_{\bm\gamma}$ has the following expression.
    \begin{equation*}
        b_{\bm\gamma}=\binom{hk}{\bm\gamma}^{-1}\sum_{\bm\alpha\in M_{k,n+1}}\binom{k}{\bm\alpha}W_{\bm\gamma,\bm\alpha}\bm X^{\bm\alpha}.
    \end{equation*}
\end{proposition}
\begin{proof}
    Since $W_{\bm q}$ is surjective, $\phi_{\bm q}$ is injective by \cref{WSurjPhiInj}, so $\S^{hk}(\KK^{m+1})\xrightarrow{\phi_{\bm q}}{\bm q}$-$\Sym_{k,n+1}$ is an isomorphism.
    To derive the explicit formula for $b_{\bm\gamma}$, let $\bm\gamma\in M_{hk,m+1}$. A direct computation shows that $W_{\bm q}^*(A_{hk}(\bm Z^{\bm\gamma}))=\langle \bm Z^{\bm\gamma},W_{\bm q}(\cdot)\rangle_{hk}$ is the linear functional on $\S^{k}(\KK^{n+1})$ given by
    \begin{equation*}
        p=\sum_{\bm\alpha\in M_{k,n+1}}p_{\bm\alpha}\bm X^{\bm\alpha}\mapsto\langle \bm Z^{\bm\gamma},W_{\bm q}(p)\rangle_{hk}=\binom{hk}{\bm\gamma}^{-1}\sum_{\bm\alpha\in M_{k,n+1}}p_{\bm\alpha} W_{\bm\gamma,\bm\alpha}.
    \end{equation*}
    Applying $A_k^{-1}$ yields the claimed expression for $b_{\bm\gamma}$.
\end{proof}

\begin{notation}
    Assume that $\phi_{\bm q}$ is injective (equivalently, $W_{\bm q}$ is surjective). We denote by $\psi_{\bm q}:{\bm q}$-$\Sym_{k,n+1}\to\S^{hk}(\KK^{m+1})$ the inverse map.
\end{notation}

The map $\psi_{\bm q}$ serves as the main tool to study the decomposition properties of ${\bm q}$-Symmetric tensors, as presented in \cref{QSymDecSubsection}.

\begin{remark}
    Using the apolarity map $A_k$ \cref{ApolarityMap}, one sees that ${\bm q}$-$\Sym_{k,n+1}^*=A_k({\bm q}$-$\Sym_{k,n+1})$.
\end{remark}
\subsection{Surjectivity of \texorpdfstring{$W_{\bm q}$}{q}}

Let ${\bm q}=(q_0,\ldots,q_n)$ be as in \cref{WeightFunction}. Since ${\bm q}$ consists of forms, we can consider the rational mapping
\begin{equation}\label{qFunction}
        \begin{aligned}
            {\bm q}:\PP^{m}(\KK)&\dashrightarrow \PP^{n}(\KK)\\
            [\bm Z]&\longmapsto [q_0(\bm Z):\cdots:q_n(\bm Z)]
        \end{aligned}
\end{equation}
and define the variety $V_{\bm q}:=\overline{\mathrm{Im}({\bm q})}$ as the Zariski closure of the image of ${\bm q}$.
The problem of finding a set ${\bm q}$ such that $W_{\bm q}$ is surjective can be reformulated in geometric terms:
\begin{proposition}\label{HilbertFuncCharacterization}
    Let ${\bm q}$ and $W_{\bm q}$ be as in \cref{WeightFunction}. Let
    \begin{equation*}
        I=\left\{p(X_0,\ldots,X_n)\in\S(\KK^{n+1})\mid p(q_0,\ldots,q_n)=0\right\}\lhd\S(\KK^{n+1})
    \end{equation*}
    be the homogeneous ideal of polynomials vanishing on the variety $V_{\bm q}$. Then $W_{\bm q}$ is surjective if and only if the Hilbert function $H_I$ of $I$ satisfies $H_I(k)=D(hk,m+1)=\binom{hk+m}{hk}$.
\end{proposition}
\begin{proof}
    Consider the substitution map induced by
    \begin{equation*}
            \begin{aligned}
                W_{\bm q}:\S(\KK^{n+1})&\longrightarrow \S(\KK^{m+1})\\
                X_i&\longmapsto q_i(\bm Z)
            \end{aligned}
    \end{equation*}
    and note that $I=\ker(W_{\bm q})$ and that $W_{\bm q}\restriction_{\S^k(\KK^{n+1})}$ coincides with the previously defined $W_{\bm q}$ in \cref{WeightFunction}. Therefore, $\S^k(\KK^{n+1})/I_k\cong \mathrm{Im}(W_{\bm q})=\mathrm{span}(\bm q^{\bm\alpha}\mid \bm\alpha\in M_{k,n+1})$ and the statement follows by dimensional count, given that $H_I(k)=\dim_\KK(\S^k(\KK^{n+1})/I_k)$.
\end{proof}

\begin{proposition}\label{NoBasePoints}
    Let ${\bm q}=(q_0,\ldots,q_n)$ with $q_i\in\S^h(\KK^{m+1})$, $k$ a positive integer and $W_{\bm q}$ be as in \cref{WeightFunction}. If $W_{\bm q}$ is surjective, then
    \begin{enumerate}
        \item The base locus of $\bm{q}$, i.e. the projective variety defined by $(q_0,\ldots,q_n)$, is empty.\label{item1TechnicalProp}
        \item The map $\bm q$ as in \cref{qFunction} is injective.
    \end{enumerate}
\end{proposition}
\begin{proof}
The proof is divided in two parts, corresponding to the two items above.
    \begin{enumerate} 
        \item We prove the contrapositive. Assume that the base locus of $\bm q$ is not empty, so there exists $\bm\xi\in\PP^m(\KK)$ such that $q_0(\bm\xi)=\ldots=q_n(\bm\xi)=0$. This implies that every form in $\mathrm{Im}(W_{\bm q})$ vanishes at $\bm\xi$, so $\mathrm{Im}(W_{\bm q})\subsetneq \S^{hk}(\KK^{m+1})$, proving that $W_{\bm q}$ is not surjective.
        \item Let $\bm \xi,\,\bm\eta\in\KK^{m+1}$ such that $[\bm q(\bm \xi)]=[\bm q(\bm\eta)]$. By \cref{item1TechnicalProp}, $\bm q(\bm \xi),\,\bm q(\bm \eta)\neq \bm 0$ and so there exists $\lambda\in\KK\setminus\{0\}$ such that $\bm q(\bm \xi)=\lambda\bm q(\bm \eta)$. Consider $p\in\S^k(\KK^{n+1})$. Evaluating the form $W_{\bm q}(p)$ at the point $\bm \xi$ yields
        \begin{equation}\label{relationWq}
            W_{\bm q}(p)(\bm \xi)=p(\bm q(\bm\xi))=p(\lambda\bm q(\bm\eta))=\lambda^kW_{\bm q}(p)(\bm\eta),
        \end{equation}
        where the last equality holds because $p$ is a form of degree $k$. From \cref{relationWq}, we deduce that every form $p'\in\mathrm{Im}(W_{\bm q})$ satisfies the equation $p'(\bm \xi)=\lambda^kp'(\bm\eta)$. By assumption, $W_{\bm q}$ is surjective, i.e. $\mathrm{Im}(W_{\bm q})=\S^{hk}(\KK^{m+1})$, so the equation is actually satisfied by any $p'$ in $\S^{hk}(\KK^{m+1})$. Assume for a contradiction that $[\bm\xi]\neq[\bm\eta]$, i.e. $\bm\xi$ and $\bm\eta$ are linearly independent as vectors in $\KK^{m+1}$. Hence, there exists a linear form $l\in\S^1(\KK^{m+1})$ such that $l(\bm\xi)=1$ and $l(\bm\eta)=0$. Note that $l^{hk}\in\S^{hk}(\KK^{m+1})$ and $1=l^{hk}(\bm\xi)\neq \lambda^kl^{hk}(\bm \eta)$, which contradicts the fact that $p'(\bm \xi)=\lambda^kp'(\bm\eta)$ for all $p'\in\S^{hk}(\KK^{m+1})$. Thus, $[\bm\xi]=[\bm\eta]$ and so $\bm q$ as in \cref{qFunction} is injective.
    \end{enumerate}
\end{proof}

\begin{proposition}\label{SurjectivityConditions}
    Let ${\bm q}=(q_0,\ldots,q_n)$ with $q_i\in\S^h(\KK^{m+1})$, $k$ a positive integer and $W_{\bm q}$ as in \cref{WeightFunction}.
    \begin{enumerate}
        \item Let ${\bm q}'=(q'_0,\ldots,q'_n)$ with $q_i'\in\S^h(\KK^{m+1})$ and $\,W_{{\bm q}'}$ as in \cref{WeightFunction}. If $\mathrm{span}({\bm q})=\mathrm{span}({\bm q}')$, then $\mathrm{span}(\bm q^{\bm\alpha}\mid\bm\alpha\in M_{k,n+1})=\mathrm{span}(\bm q'^{\bm\alpha}\mid\bm\alpha\in M_{k,n+1})$, i.e. $\mathrm{Im}(W_{\bm q})=\mathrm{Im}(W_{{\bm q}'})$.
        \item If $\{\bm q^{\bm\alpha}\mid\bm\alpha\in M_{k,n+1}\}$ is linearly independent in $\S^{hk}(\KK^{m+1})$, then $\bm q=(q_0,\ldots,q_n)$ is linearly independent in $\S^h(\KK^{m+1})$.
        \item\label{AllMonomials} If ${\bm q}=(\bm Z^{\bm\beta}\mid\bm\beta\in M_{h,m+1})$ is the standard monomial basis of $\S^h(\KK^{m+1})$, then $\{\bm q^{\bm\alpha}\mid\bm\alpha\in M_{k,n+1}\}$ is the standard monomial basis of $\S^{hk}(\KK^{m+1})$, and therefore $W_{\bm q}$ is surjective.
    \end{enumerate}
\end{proposition}

\begin{proof}
    \begin{enumerate}
        \item We show that $\mathrm{Im}(W_{\bm q})\subseteq\mathrm{Im}(W_{{\bm q}'})$. The reverse inclusion follows by exchanging the roles of $\bm q$ and $\bm q'$. Let $p\in\mathrm{Im}(W_{\bm q})=\mathrm{span}(\bm q^{\bm\alpha}\mid\bm\alpha\in M_{k,n+1})$. Since $\mathrm{span}({\bm q})\subseteq \mathrm{span}({\bm q}')$, for all $i=0,\ldots,n$, $q_i$ is a linear combination of $q'_0,\ldots,q'_n$. Furthermore, for all $\bm\alpha\in M_{k,n+1}$, the multinomial expansion of $\bm q^{\bm\alpha}$ yields that $\bm q^{\bm\alpha}\in\mathrm{span}(\bm q'^{\bm\alpha'}\mid\bm\alpha'\in M_{k,n+1})=\mathrm{Im}(W_{\bm q'})$. This implies that also $p\in\mathrm{Im}(W_{{\bm q}'})$. 
        \item We show the contrapositive. Assume $\bm q=(q_0,\ldots,q_n)$ is linearly dependent in $\S^h(\KK^{m+1})$. Without loss of generality, $q_n=\sum_{i=0}^{n-1}\lambda_iq_i$ for some scalars $\lambda_i\in\KK$, not all equal to $0$. Computing the multinomial expansion of $q_n^k$ yields
        \begin{equation*}
            q_n^k=\sum_{\bm\delta=(\delta_0,\ldots,\delta_{n-1})\in M_{k,n}}\binom{k}{\bm\delta}(\lambda_0q_0)^{\delta_0}\cdots(\lambda_{n-1}q_{n-1})^{\delta_{n-1}}.
        \end{equation*}
        Thus, $q_n^k$ is a linear combination of the forms $q_0^{\delta_0}\cdots q_{n-1}^{\delta_{n-1}}$, for $(\delta_0,\ldots,\delta_{n-1})\in M_{k,n}$, so $\{\bm q^{\bm\alpha}\mid\bm\alpha\in M_{k,n+1}\}$ is linearly dependent in $\S^{hk}(\KK^{m+1})$.
        \item Clearly $q^{\bm\alpha}$ is a standard monomial for all $\bm\alpha\in M_{k,n+1}$. Furthermore, we observe that the standard basis of $\S^{hk}(\KK^{m+1})$ is contained in $\{\bm q^{\bm\alpha}\mid\bm\alpha\in M_{k,n+1}\}$. This is because each $\bm Z^{\bm\gamma}$ (for $\bm\gamma\in M_{hk,m+1}$) is the product of $k$ monomials of the form $\bm Z^{\bm\beta}$ (for suitable $\bm\beta_1,\ldots,\bm\beta_k\in M_{h,m+1}$), hence the span of $\{\bm q^{\bm\alpha}\mid\bm\alpha\in M_{k,n+1}\}$ equals all of $\S^{hk}(\KK^{m+1})$.
    \end{enumerate}
\end{proof}

\begin{corollary}\label{CorollarySurjectivity}
    Let ${\bm q}$ and $W_{\bm q}$ be as in \cref{WeightFunction}. If ${\bm q}$ is a basis for $\S^h(\KK^{m+1})$, then $W_{\bm q}$ is surjective.
\end{corollary}
\begin{proof}
    The statement is an immediate consequence of \cref{SurjectivityConditions} (1), (3).
\end{proof}

The following is a counterexample to the converse of \cref{SurjectivityConditions} (1) and \cref{CorollarySurjectivity}.

\begin{example}\label[example]{ExampleHenri}
    Consider $k=3$ and
    \begin{equation*}
        \bm q=(Z_0^3,Z_0^2Z_1,Z_0^2Z_2,Z_0Z_1^2,Z_0Z_2^2,Z_1^3,Z_1^2Z_2,Z_1Z_2^2,Z_2^3,2Z_2^3).
    \end{equation*}
    Observe that the first $9$ monomials of $\bm q$ are linearly independent, and the last monomial is twice the previous one, thus $\dim(\mathrm{span}(\bm q))=9<10=D(3,3)=\dim(\S^3(\KK^3))$. So, $\bm q$ does not span $\S^3(\KK^3)$. Nonetheless, one checks directly that $W_{\bm q}:\S^{3}(\KK^{10})\to\S^{9}(\KK^3)$ is surjective:
    \begin{align*}
        W_{\bm q}(X_0^3)&=Z_0^9; & W_{\bm q}(X_0^2X_1) &=Z_0^8Z_1; & W_{\bm q}(X_0^2X_2)&=Z_0^8Z_2;\\
        W_{\bm q}(X_0^2X_3)&=Z_0^7Z_1^2; & W_{\bm q}(X_0X_1X_2)&=Z_0^7Z_1Z_2; & W_{\bm q}(X_0^2X_4) &=Z_0^7Z_2^2; \\
        W_{\bm q}(X_0^2X_5)&=Z_0^6Z_1^3; & W_{\bm q}(X_0^2X_6)&=Z_0^6Z_1^2Z_2; & W_{\bm q}(X_0^2X_7)&=Z_0^6Z_1Z_2^2;\\
        W_{\bm q}(X_0^2X_8) &=Z_0^6Z_2^3; & W_{\bm q}(X_0X_1X_5)&=Z_0^5Z_1^4; & W_{\bm q}(X_0X_2X_5)&=Z_0^5Z_1^3Z_2;\\
        W_{\bm q}(X_0X_1X_8)&=Z_0^5Z_1Z_2^3; & W_{\bm q}(X_0X_2X_8) &=Z_0^5Z_2^4; & W_{\bm q}(X_0X_3X_5)&=Z_0^4Z_1^5; \\
        W_{\bm q}(X_0X_3X_6)&=Z_0^4Z_1^4Z_2; & W_{\bm q}(X_0X_4X_5)&=Z_0^4Z_1^3Z_2^2; & W_{\bm q}(X_0X_3X_8) &=Z_0^4Z_1^2Z_2^3; \\
        W_{\bm q}(X_0X_4X_7)&=Z_0^4Z_1Z_2^4; & W_{\bm q}(X_0X_4X_8)&=Z_0^4Z_2^5; & W_{\bm q}(X_0X_5^2)&=Z_0^3Z_1^6;\\
        W_{\bm q}(X_0X_5X_6) &=Z_0^3Z_1^5Z_2; & W_{\bm q}(X_0X_6^2)&=Z_0^3Z_1^4Z_2^2; & W_{\bm q}(X_0X_5X_8)&=Z_0^3Z_1^3Z_2^3.
    \end{align*}
    The remaining monomials are obtained similarly. Let $\bm q'$ be the standard monomial basis of $\S^3(\KK^3)$, i.e.
    \begin{equation*}
        \bm q'=(Z_0^3,Z_0^2Z_1,Z_0^2Z_2,Z_0Z_1^2,Z_0Z_1Z_2,Z_0Z_2^2,Z_1^3,Z_1^2Z_2,Z_1Z_2^2,Z_2^3).
    \end{equation*}
    By \cref{SurjectivityConditions} (3), $W_{\bm q'}$ is also surjective, so on the one hand $\mathrm{Im}(W_{\bm q})=\mathrm{Im}(W_{\bm q'})=\S^9(\KK^3)$. On the other hand, $\mathrm{span}(\bm q)\subsetneq\mathrm{span}(\bm q')$. Hence, $\bm q$ and $\bm q'$ are a counterexample to the converse of \cref{SurjectivityConditions} (1). Furthermore, $W_{\bm q}$ is surjective, while $\bm q$ is not a basis of $\S^3(\KK^3)$, so $\bm q$ is a counterexample to the converse of \cref{CorollarySurjectivity}.
\end{example}
In geometric terms, \cref{ExampleHenri} also shows that surjectivity of $W_{\bm q}$ gives no information about whether the variety $V_{\bm q}$ is degenerate or not. In fact, both $W_{\bm q}$ and $W_{\bm q'}$ are surjective, but $V_{\bm q}$ is degenerate, while $V_{\bm q'}$ is not.
The following is a counterexample to the converse of \cref{SurjectivityConditions} (2).

\begin{example}
    We show that the assumption that $\bm q$ is the standard monomial basis in \cref{SurjectivityConditions} (3) cannot be weakened to $\bm q$ is a basis of $\S^h(\KK^{m+1})$. Let $n=h=k=2$ and $m=1$, $\dim(\S^2(\KK^2))=\binom{3}{2}=3$, $\dim(\S^4(\KK^2))=\binom{5}{4}=5$. The standard basis $\bm q=(Z_0^2,Z_0Z_1,Z_1^2)$ of $\S^2(\KK^2)$ gives the standard basis $\{\bm q^{\bm\alpha}\mid \bm\alpha\in M_{2,3}\}=\{Z_0^4,Z_0^3Z_1,Z_0^2Z_1^2,Z_0Z_1^3,\allowbreak Z_1^4\}$ of $\S^{4}(\KK^2)$. Let $\bm q'=(Z_0^2,Z_0Z_1+Z_1^2,Z_1^2)$ be another basis of $\S^2(\KK^2)$. We compute $\{\bm q'^{\bm\alpha}\mid\bm\alpha\in M_{2,3}\}=\{Z_0^4,Z_0^3Z_1+Z_0^2Z_1^2,Z_0^2Z_1^2,Z_0^2Z_1^2+2Z_0Z_1^3+Z_1^4,Z_0Z_1^3+Z_1^4,Z_1^4\}$, which consists of $6$ forms, and therefore, is not a basis of $\S^4(\KK^2)$, since $\dim(\S^4(\KK^2))=5$.
\end{example}

\subsection{\texorpdfstring{$\bm q$}{q}-Symmetric decompositions}\label{QSymDecSubsection}
Let ${\bm q}=(q_0,\ldots,q_n)$ be as in \cref{WeightFunction}.
\begin{definition}
    Let $p\in \S^k(\KK^{n+1})$. A \emph{${\bm q}$-Symmetric decomposition} of $p$ over $\KK$ is an expression of the form
    \begin{equation}\label{QSymDecomposition}
        p=\sum_{i=1}^r \lambda_i\langle \bm q(\bm\xi_i),\bm X\rangle^{k}:=\sum_{i=1}^r \lambda_i\left(q_0(\bm\xi_i)X_0+\cdots+q_n(\bm\xi_i)X_n\right)^{k},
    \end{equation}
    where $r\in\NN,\lambda_1,\ldots,\lambda_r\in\KK\setminus\{0\}$ and $\bm\xi_1,\ldots,\bm\xi_r\in\KK^{m+1}$ correspond to different points in $\PP^{m}(\KK)$.
\end{definition}

\begin{remark}\label[remark]{RemarkWaringqSym}
The connection between ${\bm q}$-Symmetric decompositions and standard Waring decomposition can be summarized as follows.
\begin{enumerate}
    \item  A ${\bm q}$-Symmetric decomposition \labelcref{QSymDecomposition} is, under the surjectivity hypothesis on $W_{\bm q}$ (cf. \cref{NoBasePoints}), in particular, a Waring decomposition \labelcref{WaringDecomposition} with nodes $\bm q(\bm\xi_i)=(q_0(\bm\xi_i),\ldots,q_n(\bm\xi_i))$, for $i=1,\ldots,r$.
    \item Conversely, \labelcref{WaringDecomposition} is a special case of \labelcref{QSymDecomposition}, when $n=m$ and $q_i(X_0,\ldots,X_n)=X_i$.
\end{enumerate}
\end{remark}

\begin{theorem}\label{QSymMainTheorem}
Suppose that $W_{\bm q}$ is surjective.
The following statements hold:
\begin{enumerate}
    \item There is a bijective correspondence between length $r$ ${\bm q}$-Symmetric decompositions
    \begin{equation*}
        p=\sum_{i=1}^r\lambda_i\langle \bm q(\bm\xi_i),\bm X\rangle^k
    \end{equation*}
    of $p \in{\bm q}$-$\Sym_{k,n+1}$ over $\KK$ and length $r$ Waring decompositions
    \begin{equation*}
        \psi_{\bm q}(p)=\sum_{i=1}^r\lambda_i\langle\bm \xi_{i},\bm Z\rangle^{hk}
    \end{equation*}
    of $\psi_{\bm q}(p)$ over $\KK$, where $\lambda_1,\ldots,\lambda_r\in\KK\setminus\{0\}$ and $\bm\xi_1,\ldots,\bm\xi_r\in\KK^{m+1}$ correspond to different points in $\PP^{m}(\KK)$.
    \item The minimal length of a ${\bm q}$-Symmetric decomposition of $p\in{\bm q}$-$\Sym_{k,n+1}$ over $\KK$ is equal to the Waring rank of $\psi_{\bm q}(p)$ over $\KK$.
\end{enumerate}
\end{theorem}

\begin{proof}
    Let $p\in{\bm q}$-$\Sym_{k,n+1}$. We show that the correspondence between ${\bm q}$-Symmetric decompositions of $p$ and Waring decompositions of $\psi_{\bm q}(p)$ is bijective. Observe that
    \begin{align*}
        \psi_{\bm q}(p)&=\sum_{i=1}^r \lambda_i\left(\xi_{i,0} Z_0+\cdots+\xi_{i,m}Z_m\right)^{hk}=\sum_{\bm\gamma\in M_{hk,m+1}}\binom{hk}{\bm\gamma}\left(\sum_{i=1}^r \lambda_i\bm\xi_i^{\bm\gamma}\right)\bm Z^{\bm\gamma}
    \end{align*}
    if and only if
    \begin{align*}
        p&=\sum_{\bm\alpha\in M_{k,n+1}}\binom{k}{\bm\alpha}\left(\sum_{\bm\gamma\in M_{hk,m+1}}\binom{hk}{\bm\gamma}^{-1}\binom{hk}{\bm\gamma}\left(\sum_{i=1}^r\lambda_i\bm\xi_{i}^{\bm\gamma}\right)W_{\bm\gamma,\bm\alpha}\right)\bm X^{\bm\alpha}\\
        &=\sum_{\bm\alpha\in M_{k,n+1}}\binom{k}{\bm\alpha}\left(\sum_{\bm\gamma\in M_{hk,m+1}}\left(\sum_{i=1}^r\lambda_i\bm\xi_{i}^{\bm\gamma}\right)W_{\bm\gamma,\bm\alpha}\right)\bm X^{\bm\alpha}\\
        &=\sum_{\bm\alpha\in M_{k,n+1}}\binom{k}{\bm\alpha}\left(\sum_{i=1}^r\lambda_i\left(\sum_{\bm\gamma\in M_{hk,m+1}}\bm\xi_{i}^{\bm\gamma}W_{\bm\gamma,\bm\alpha}\right)\right)\bm X^{\bm\alpha} =\sum_{\bm\alpha\in M_{k,n+1}}\binom{k}{\bm\alpha}\left(\sum_{i=1}^r\lambda_i\bm q(\bm\xi_i)^{\bm\alpha}\right)\bm X^{\bm\alpha}\\
        &=\sum_{i=1}^r\lambda_i\left(\sum_{\bm\alpha\in M_{k,n+1}}\binom{k}{\bm\alpha}\bm q(\bm\xi_i)^{\bm\alpha}\bm X^{\bm\alpha}\right) =\sum_{i=1}^r\lambda_i\langle \bm q(\bm\xi_i),\bm X\rangle^k
    \end{align*}
    Finally, the computation above also shows that the bijective correspondence between ${\bm q}$-Sym\-met\-ric decompositions of $p$ and Waring decompositions of $\psi_{\bm q}(p)$ is length preserving, unless at least one of $\bm\xi_1,\ldots,\bm\xi_r$ is a base point of $\bm q$, or $[\bm q(\bm\xi_i)]=[\bm q(\bm\xi_j)]$ for $i\neq j$, but, when $W_{\bm q}$ is surjective, this is never the case, due to \cref{NoBasePoints}.
\end{proof}
Under the assumption of surjectivity of $W_{\bm q}$, we can characterize all the forms $p\in\S^k(\KK^{n+1})$ admitting a Waring decomposition on $V_{\bm q}$.
\begin{theorem}\label{QSymOnlyDec}
    Suppose that $W_{\bm q}$ is surjective and let $p\in\S^k(\KK^{n+1})$. Then $p$ admits a Waring decomposition on $V_{\bm q}$ if and only if $p\in{\bm q}$-$\Sym_{k,n+1}$.
\end{theorem}
\begin{proof}
    If $p\in{\bm q}$-$\Sym_{k,n+1}$, since $W_{\bm q}$ is surjective, \cref{QSymMainTheorem} guarantees the existence of a ${\bm q}$-Symmetric decomposition, as $\psi_{\bm q}(p)$ always admits a Waring decomposition. Thus, since a ${\bm q}$-Symmetric decomposition is a Waring decomposition on $\mathrm{Im}(\bm q)\subseteq V_{\bm q}$, we conclude that $\bm q$-Symmetric tensors have Waring decompositions on $V_{\bm q}$.

    Now assume that $p\in\S^k(\KK^{n+1})$ has a ${\bm q}$-Symmetric decomposition
    \begin{equation*}
        p=\sum_{i=1}^r\lambda_i\langle \bm q(\bm\xi_i),\bm X\rangle^k.
    \end{equation*}
    Expanding this expression yields
    \begin{align*}
        p& = \sum_{i=1}^r\lambda_i\sum_{\bm\alpha\in M_{k,n+1}}\binom{k}{\bm\alpha} \bm q(\bm\xi_i)^{\bm\alpha}\bm X^{\bm\alpha} =\sum_{i=1}^r\lambda_i\sum_{\bm\alpha\in M_{k,n+1}}\binom{k}{\bm\alpha} \bm X^{\bm\alpha}\sum_{\bm\gamma\in M_{hk,m+1}}W_{\bm\gamma,\bm\alpha}\bm\xi_i^{\bm\gamma}\\
        &=\sum_{\bm\gamma\in M_{hk,m+1}}\sum_{i=1}^r\lambda_i\bm\xi_i^{\bm\gamma}\binom{hk}{\bm\gamma}\left(\binom{hk}{\bm\gamma}^{-1}\sum_{\bm\alpha\in M_{k,n+1}}\binom{k}{\bm\alpha} \bm X^{\bm\alpha}W_{\bm\gamma,\bm\alpha}\right)=\sum_{\bm\gamma\in M_{hk,m+1}}\sum_{i=1}^r\lambda_i\bm\xi_i^{\bm\gamma}\binom{hk}{\bm\gamma}b_{\bm\gamma},
    \end{align*}
    where $\B:=\left\{b_{\bm\gamma}=\binom{hk}{\bm\gamma}^{-1}\sum_{\bm\alpha\in M_{k,n+1}}\binom{k}{\bm\alpha} \bm X^{\bm\alpha}W_{\bm\gamma,\bm\alpha}\right\}$ is the standard basis of ${\bm q}$-$\Sym_{k,n+1}$.
    This proves that $p$ is a linear combination of the basis elements in $\B$, and thus $p\in{\bm q}$-$\Sym_{k,n+1}$.
    From the implications above, we deduce that ${\bm q}$-$\Sym_{k,n+1}$ is generated by $\{l_{\bm q(\bm\xi)}^k=(q_0(\bm\xi)X_0+\cdots +q_n(\bm\xi)X_n)^k\in\S^k(\KK^{n+1})\mid[q_0(\bm\xi):\cdots:q_n(\bm\xi)]\in \mathrm{Im}(\bm q)\}$.
    Conversely, note that any form $f\in {\bm q}$-$\Sym_{k,n+1}^{\perp}$, the orthogonal complement of ${\bm q}$-$\Sym_{k,n+1}$ with respect to the apolar product, vanishes on $\mathrm{Im}(\bm q)$, as a consequence of bilinearity of the inner product and
    \begin{equation*}
        0=\langle f,l_{\bm q(\bm\xi)}^k\rangle_k=\langle f,(q_0(\bm\xi)X_0+\cdots+ q_n(\bm\xi)X_n)^k\rangle_k=f(\bm q(\bm\xi))
    \end{equation*}
    for any $[q_0(\bm\xi):\cdots:q_n(\bm\xi)]\in \mathrm{Im}(\bm q)$. Furthermore, any polynomial vanishing on $\mathrm{Im}(\bm q)$ also vanishes on the Zariski closure $\overline{\mathrm{Im}(\bm q)}=V_{\bm q}$, hence $0=\langle f,l_{\bm v}^k\rangle_k=f(\bm v)$ for any $[v_0:\cdots:v_n]\in V_{\bm q}$. Thus, $l_{\bm v}^k$ is orthogonal to every form which is orthogonal to ${\bm q}$-$\Sym_{k,n+1}$. Since the apolar product is non-degenerate, we conclude that $l_{\bm v}^k\in {\bm q}$-$\Sym_{k,n+1}$. By linearity, it follows that any form $p\in\S^k(\KK^{n+1})$ admitting a Waring decomposition on $V_{\bm q}$ is a $\bm q$-Symmetric tensor, i.e. $p\in {\bm q}$-$\Sym_{k,n+1}$, completing the proof of the converse implication.
\end{proof}

In the remainder of the section, we set $\KK=\RR$ and study \emph{positive} ${\bm q}$-Symmetric decompositions of ${\bm q}$-Symmetric tensors.
\begin{definition}
    Let $p\in{\bm q}$-$\Sym_{k,n+1}$ be a non-zero form. A \emph{positive} ${\bm q}$-Symmetric decomposition of $p$ is a ${\bm q}$-Symmetric decomposition \begin{equation*}
        p=\sum_{i=1}^r \lambda_i\langle \bm q(\bm\xi_i),\bm X\rangle^{k},
    \end{equation*}
    where $r\in\NN,\lambda_1,\ldots,\lambda_r>0$ and $\bm\xi_1,\ldots,\bm\xi_r\in\RR^{m+1}$ correspond to different points in $\PP^{m}(\RR)$. We call positive ${\bm q}$-Symmetric tensors those ${\bm q}$-Symmetric tensors admitting a positive ${\bm q}$-Symmetric decomposition.
\end{definition}

We now restrict to the case of $k$ even and characterize the dual cone of the cone of \emph{positive ${\bm q}$-Symmetric tensors}.
\begin{theorem}[Dual of the cone of positive ${\bm q}$-Symmetric tensors]\label{DualCone}
    Let $\Sigma_{\bm q}\subseteq \S^k(\RR^{n+1})$ be the cone of positive ${\bm q}$-Symmetric tensors. The cone $\Sigma_{\bm q}^*$ dual to $\Sigma_{\bm q}$ with respect to the apolar product is
    \begin{equation*}
        \Sigma_{\bm q}^*=\left\{p\in\S^k(\RR^{n+1})\mid p(\bm q(\bm\xi))\geq0 \text{ for all }\bm\xi\in\mathbb{S}^{m}\right\}.
    \end{equation*}
\end{theorem}
\begin{proof}
    By bilinearity of the apolar product, a form $p\in \S^k(\RR^{n+1})$ lies in the dual cone $\Sigma_{\bm q}^*$ if and only if, for all $\bm\xi\in\mathbb{S}^{m}$ it holds
    \begin{equation}\label{ApolarProductEvaluation}
        \langle p,(q_0(\bm\xi)X_0+\cdots+ q_n(\bm\xi)X_n)^k\rangle_k\geq0,
    \end{equation}
    where, without loss of generality, we consider $\bm\xi\in\mathbb{S}^{m}$ instead of $\bm\xi\in\RR^{m+1}$ because the polynomials involved are homogeneous and $k$ is even. Finally, note that, by straightforward computation, the left-hand side of \cref{ApolarProductEvaluation} is equal to the evaluation $p(\bm q(\bm\xi))$, completing the proof.
\end{proof}

\begin{proposition}\label{prop:waringIFFcoeffsmomentsequence}
    Let $p=\sum_{\bm\alpha\in M_{k,n+1}}\binom{k}{\bm\alpha}p_{\bm\alpha}\bm X^{\bm\alpha}\in\S^{k}(\KK^{n+1})$ be a form, $A\subseteq\KK^{n+1}$ a set. There is a length-preserving bijective correspondence between positive Waring decompositions of $p$ of length $r$ on $A$
    \begin{equation*}
        p=\sum_{i=1}^r\lambda_i\langle \bm\xi_i,\bm X\rangle^k,
    \end{equation*}
    and $r$-atomic representing measures $\sum_{i=1}^r\lambda_i\delta_{\bm\xi_i}$ for $(p_{\bm\alpha})_{\bm\alpha\in M_{k,n+1}}$ supported on $A$, i.e. $p_{\bm\alpha}=\sum_{i=1}^r\lambda_i\bm\xi_i^{\bm\alpha}$ for all $\bm\alpha\in M_{k,n+1}$,
    where $r\in\NN$, $\lambda_1,\ldots,\lambda_r>0$ and $\bm\xi_1,\ldots,\bm\xi_r\in A$.
\end{proposition}
\begin{proof}
    Expanding the positive Waring decomposition of $p$ yields
    \begin{align*}
        p&=\sum_{\bm\alpha\in M_{k,n+1}}\binom{k}{\bm\alpha}p_{\bm\alpha}\bm X^{\bm\alpha}=\sum_{i=1}^r\lambda_i\left(\xi_{i,0} X_0 + \cdots + \xi_{i,n} X_n\right)^{k}=\sum_{i=1}^r\lambda_i\left(\sum_{\bm\alpha\in M_{k,n+1}}\binom{k}{\bm\alpha}\bm\xi_i^{\bm\alpha}\bm X^{\bm\alpha}\right)\\
        &=\sum_{\bm\alpha\in M_{k,n+1}}\binom{k}{\bm\alpha}\left(\sum_{i=1}^r\lambda_i\bm\xi_i^{\bm\alpha}\right)\bm X^{\bm\alpha}.
    \end{align*}
    Comparing coefficients in $p$ yields the statement.
\end{proof}

In the next proposition, we prove a sufficient condition, based on the theory of flat extensions of moment matrices (see \cite{curtoFialkowTruncatedMP}), for a $\bm q$-Symmetric tensor to have a positive $\bm q$-Symmetric decomposition.

\begin{proposition}\label{FlatExtensionForDecomps}
    Let $k$ be even. Suppose that $W_{\bm q}$ is surjective and let $p\in{\bm q}$-$\Sym_{k,n+1}$ and $\psi_{\bm q}(p)=\sum_{\bm\gamma\in M_{hk,m+1}}\binom{hk}{\bm\gamma}c_{\bm\gamma}\bm Z^{\bm\gamma}$, $H=\operatorname{Hankel}(\bm c)$ be the Hankel matrix of $\bm c=(c_{\bm\gamma})_{\bm\gamma\in M_{hk,m+1}}$, with $r=\rank(H)$.
    If $H$ is PSD and $\rank(H_0)=r$, where $H_0$ is the principal submatrix of $H$ indexed by the multiples of $Z_0$ , then $p$ has a positive ${\bm q}$-Symmetric decomposition of length $r=\rank(H)$. Conversely, if $p$ has a positive ${\bm q}$-Symmetric decomposition of length $s$, then $H$ is PSD and $\rank(H)\leq s$.
\end{proposition}
\begin{proof}
    Suppose that $\rank(H_0)=\rank(H)$. By \cref{QSymMainTheorem}, the statement reduces to showing that $\psi_{\bm q}(p)$ has a positive Waring decomposition of length $r=\rank(H)$. $H$ is identified with the quadratic form on $\S^{hk/2}(\RR^{m+1})$ having $H$ as the associated matrix with respect to the standard monomial basis $\{\bm Z^{\bm\eta}\mid \bm\eta\in M_{hk/2,m+1}\}$. Via dehomogenization, setting $Z_0=1$, $H$ is now indexed by all monomials of degree at most $hk/2$, in the $m$ variables $Z_1,\ldots,Z_m$. We now apply results on the theory of flat extensions of moment matrices to the matrix $H$.
    Observe that the principal submatrix $H_0$ of $H$ that, prior to dehomogenization, was indexed by the multiples of $Z_0$, after dehomogenization corresponds to the matrix indexed by exactly all monomials of degree at most $hk/2-1$, in the $m$ variables $Z_1,\ldots,Z_m$. With the goal of applying \cite[Thm. 5.29]{LaurentSOSMomentOpt}, note that $H_0$ is the truncated moment matrix of the tuple $\bm c$ of order $hk/2-1$ and $H$ is the truncated moment matrix of the tuple $\bm c$ of order $hk/2$. By assumption, the two matrices have the same rank and $H$ is PSD. Hence, by \cite[Thm. 5.29]{LaurentSOSMomentOpt}, $\bm c$ is the truncated moment sequence of an $r$-atomic measure $\sum_{i=1}^r\lambda_i\delta_{\bm\xi_i}$, where $\lambda_1,\ldots,\lambda_r>0$, $\bm\xi_1,\ldots,\bm\xi_r\in\RR^m$, $\delta_{\bm\xi_i}$ is the Dirac measure centered in $\bm\xi_i$ and, by \cite[Thm. 5.30]{LaurentSOSMomentOpt}, $r=\rank(H)$. Equivalently,
    \begin{equation}\label{pgammaExpression}
         c_{\bm\gamma}=\sum_{i=1}^r\lambda_i(1,\xi_{i,1},\ldots,\xi_{i,m})^{\bm\gamma}=\sum_{i=1}^r\lambda_i\prod_{j=1}^m\xi_{i,j}^{\gamma_j}, \quad\text{for all }\bm\gamma=(\gamma_0,\ldots,\gamma_m)\in M_{hk,m+1}.
    \end{equation}
    Finally, by \cref{prop:waringIFFcoeffsmomentsequence}, \cref{pgammaExpression} is equivalent to $\psi_{\bm q}(p)$ having the positive Waring decomposition $\psi_{\bm q}(p)=\sum_{i=1}^r \lambda_i\langle (1,\xi_{i,1},\ldots,\xi_{i,m}),\bm Z\rangle^{hk}$.
    By \cref{QSymMainTheorem}, we conclude that $p$ has the positive $\bm q$-Symmetric decomposition $p=\sum_{i=1}^r \lambda_i\langle \bm q(1,\xi_{i,1},\ldots,\xi_{i,m}),\bm X\rangle^{k}$.

    Conversely, if $p$ has the positive $\bm q$-Symmetric decomposition $p=\sum_{i=1}^s \lambda_i\langle \bm q(\bm\xi_i),\bm X\rangle^{k}$, where $s\in\NN,\,\lambda_1,\ldots,\lambda_s>0$ and $\bm\xi_1,\ldots,\bm\xi_s\in\RR^{m+1}$ correspond to different points in $\PP^{m}(\RR)$, then by \cref{QSymMainTheorem} we have $\psi_{\bm q}(p)=\sum_{i=1}^s\lambda_i\langle\bm \xi_{i},\bm Z\rangle^{hk}$. By \cref{prop:waringIFFcoeffsmomentsequence}, it holds that $c_{\bm\gamma}=\sum_{i=1}^s\lambda_i\bm\xi_i^{\bm\gamma}$ for all $\bm\gamma\in M_{hk,m+1}$. Observe that for all $\bm v=(v_{\bm\eta})_{\bm\eta\in M_{hk/2,m+1}}\in \RR^{D(hk/2,m+1)}$, we have
    \begin{equation*}
        \bm v^\top  H\bm v=\hspace{-0.2cm}\sum_{\bm\eta,\bm\theta\in M_{hk/2,m+1}}\hspace{-0.4cm}v_{\bm\eta}v_{\bm\theta}c_{\bm\eta+\bm\theta}=\hspace{-0.2cm}\sum_{\bm\eta,\bm\theta\in M_{hk/2,m+1}}\hspace{-0.4cm}v_{\bm\eta}v_{\bm\theta}\left(\sum_{i=1}^s\lambda_i\bm\xi_i^{\bm\eta+\bm\theta}\right)=\sum_{i=1}^s\lambda_i\Big(\hspace{-0.2cm}\sum_{\bm\eta\in M_{hk/2,m+1}}\hspace{-0.3cm}v_{\bm\eta}\bm\xi_i^{\bm\eta}\Big)^2\geq0,
    \end{equation*}
    proving that $H$ is PSD. Finally, since
    \begin{equation*}
        H=\left(\sum_{i=1}^s\lambda_i\bm\xi_i^{\bm\eta+\bm\theta}\right)_{\bm\eta,\bm\theta\in M_{hk/2,m+1}},
    \end{equation*}
    we can define $\bm u_i=(\bm\xi_i^{\bm\eta})_{\bm\eta\in M_{hk/2,m+1}}\in\RR^{D(hk/2,m+1)}$ and it is clear that $H=\sum_{i=1}^s\lambda_i\bm u_i\bm u_i^\top $, so $\rank(H)\leq s$.
\end{proof}
In the proof of \cref{FlatExtensionForDecomps}, we used \cref{QSymMainTheorem} and reduced the problem to showing that, under the flatness condition $\rank(H)=\rank(H_0)$, positive-semidefiniteness of $H$ implies that $\psi_{\bm q}(p)$ has a positive Waring decomposition. Conversely, any positive Waring decomposition of $\psi_{\bm q}(p)$ implies that $H$ is PSD. It should be noted that, in the binary case (i.e. for $m=1$), a proof of this result can be found in \cite[Corollary 6.14(ii).]{reznick-SOEP}. Furthermore, a similar result to \cite[Thm. 5.30]{LaurentSOSMomentOpt}, in the binary case, is also proved in \cite[Thm. 4.6]{reznick-SOEP} using different techniques.

\subsection{Waring decompositions on toric varieties}
In this section, we study the special case where ${\bm q}$ consists exclusively of monomials. We refer to such forms as monomial $\bm q$-Symmetric tensors. Observe that in this case $V_{\bm q}$, parametrized by a monomial map, is a toric variety and $\bm q$-Symmetric decompositions correspond to Waring decompositions on the toric variety $V_{\bm q}$. Let ${\bm q}=(\bm Z^{\bm\beta_i}\mid\bm\beta_i\in M_{h,m+1}\text{ for }i=0,\ldots,n)$. In this setting, the weight function becomes
\begin{equation}\label{WeightFunctionMonomials}
        \begin{aligned}
            W_{\bm q}:\S^k(\KK^{n+1})&\longrightarrow \S^{hk}(\KK^{m+1})\\
            \bm X^{\bm \alpha}&\longmapsto W_{\bm q}(\bm X^{\bm \alpha}):=\prod_{i=0}^n (\bm Z^{\bm\beta_i})^{\alpha_i}
        \end{aligned}
    \end{equation}
Let $\W=(W_{\bm\gamma_i,\bm\alpha_j})_{1\leq i\leq t,1\leq j\leq s}$ be the matrix associated to $W_{\bm q}$ with respect to the monomial bases, i.e. $W_{\bm q}(\bm X^{\bm\alpha})=\sum_{\bm\gamma\in M_{hk,m+1}}W_{\bm\gamma,\bm\alpha} \bm Z^{\bm\gamma}$. For this choice of ${\bm q}$, the matrix entries are given by
\begin{equation*}
    W_{\bm\gamma,\bm\alpha}=
    \begin{cases*}
      $1$ & if $\bm\gamma=\alpha_0\bm\beta_0+\cdots+\alpha_n\bm\beta_n$ \\
      $0$        & otherwise
    \end{cases*}
    \quad\text{for all }\bm\alpha\in M_{k,n+1},\;\bm\gamma\in M_{hk,m+1}.
  \end{equation*}
Defining $\phi_{\bm q}:=A_k^{-1}\circ W_{\bm q}^*\circ A_{hk}$, the set $\B$ takes the form
\begin{equation*}
    \B:=\Bigg\{b_{\bm\gamma}=\binom{hk}{\bm\gamma}^{-1}\sum_{\substack{\bm\alpha\in M_{k,n+1}\\\alpha_0\bm\beta_0+\cdots+\alpha_n\bm\beta_n=\bm\gamma}}\binom{k}{\bm\alpha}\bm X^{\bm\alpha}\mid\bm\gamma\in M_{hk,m+1}\Bigg\}
\end{equation*}
and is again linearly independent under surjectivity of $W_{\bm q}$.
\begin{definition}
    Let $p\in \S^k(\KK^{n+1})$. When ${\bm q}$ consists of monomials $(\bm Z^{\bm\beta_i}\mid\bm\beta_i\in M_{h,m+1}\text{ for }i=0,\ldots,n)$, a ${\bm q}$-Symmetric decomposition of $p$ over $\KK$ takes the form
    \begin{equation*}
        p=\sum_{i=1}^r\lambda_i\Big(\sum_{j=0}^n\bm\xi_i^{\bm\beta_j}X_j\Big)^k,
    \end{equation*}
    for some $r\in\NN,\lambda_1,\ldots,\lambda_r\in\KK\setminus\{0\}$ and $\bm\xi_1,\ldots,\bm\xi_r\in\KK^{m+1}$ corresponding to different points in $\PP^{m}(\KK)$ and it is called \emph{monomial $\bm q$-Symmetric decomposition} over $\KK$.
\end{definition}
For monomial $\bm q$-Symmetric tensors we obtain the following corollary of \cref{QSymMainTheorem}.
\begin{corollary}\label{QSymMonomialMainTheorem}
    Assume $W_{\bm q}$ is surjective. The following statements hold:
    \begin{enumerate}
    \item There is a one-to-one correspondence between length $r$ monomial $\bm q$-Symmetric decompositions
    \begin{equation*}
        p=\sum_{i=1}^r\lambda_i\Big(\sum_{j=0}^n \bm\xi_i^{\bm\beta_j}X_j\Big)^k
    \end{equation*}
    of $p\in \bm q\text{-Sym}_{k,n+1}$, over $\KK$ and length $r$ Waring decompositions
    \begin{equation*}
        \psi_{\bm q}(p)=\sum_{i=1}^r\lambda_i\langle\bm \xi_{i},\bm Z\rangle^{hk}
    \end{equation*}
    of $\psi_{\bm q}(p)$ over $\KK$, where $\lambda_1,\ldots,\lambda_r\in\KK\setminus\{0\}$ and
    $\bm\xi_1,\ldots,\bm\xi_r\in\KK^{m+1}$ correspond to different points in $\PP^{m}(\KK)$.
    \item The minimal length of a monomial $\bm q$-Symmetric decomposition of  any $p\in \bm q\text{-Sym}_{k,n+1}$ over $\KK$ is equal to the Waring rank of $\psi_{\bm q}(p)$ over $\KK$.
\end{enumerate}
\end{corollary}

We illustrate the definitions above with two examples where \cref{SurjectivityConditions} item~\ref{AllMonomials} applies, so the corresponding $W_{\bm q}$, as in \cref{WeightFunctionMonomials}, is surjective.
\begin{examples}
    \begin{enumerate}
        \item Let $m=1, h=n$ and $q_i(Z_0,Z_1)=Z_0^{n-i}Z_1^i$ be the elements of the standard monomial basis, so $W_{\bm q}(\bm X^{\bm\alpha}):=Z_0^{nk-\sum_{i=0}^n i\alpha_i}Z_1^{\sum_{i=0}^n i\alpha_i}$ and $W_{\bm q}$ is surjective. In this case $\B$ simplifies to
        \begin{equation*}
            \B=\left\{b_j=\binom{nk}{j}^{-1}\sum_{\sum_{i=0}^n i\alpha_i=j}\binom{k}{\bm\alpha}\bm X^{\bm\alpha}\mid j=0,\ldots,nk\right\}
        \end{equation*}
        Moreover, the space $\bm q\text{-Sym}_{k,n+1}=\mathrm{span}(\B)$ is called \emph{vector space of Hankel tensors} in \cite{qiHankelTensors}. In particular, if $p=\sum_{j=0}^{nk}c_jb_j\in\bm q\text{-Sym}_{k,n+1}$, the vector $\bm c=(c_0,\ldots,c_{nk})$ is $\bm v$ in the notation of \cite[(1.1)]{qiHankelTensors}. Furthermore, the binary form $\psi_{\bm q}(p)$ corresponds to the \emph{associated plane tensor of $p$} in the terminology of \cite{qiHankelTensors}.
        \item Let $m=2,\;n=5,\;k=2,\; h=2$ and ${\bm q}$ as follows
        \begin{equation*}
            \begin{gathered}
                q_0=Z_0Z_1;\quad q_1=Z_1^2;\quad q_2=Z_0Z_2;\quad
                q_3=Z_2^2;\quad q_4=Z_1Z_2;\quad q_5=Z_0^2.
            \end{gathered}
        \end{equation*}
        For this choice of ${\bm q}$, surjectivity of $W_{\bm q}$ follows from \cref{AllMonomials} of \cref{SurjectivityConditions}. Furthermore, the basis $\B=\{b_i\in\S^2(\KK^{6})\mid i=1,\ldots,15\}$ of $\bm q\text{-Sym}_{2,6}$ has the following expression.
        \begin{align*}
            b_1 &= X_5^2, & b_2 &=\frac{X_0X_5}{2}, & b_3 &=\frac{X_2X_5}{2},\\
            b_4 &= \frac{X_0^2+2X_1X_5}{6}, & b_5 &=\frac{X_0X_2+X_4X_5}{6}, & b_6 &=\frac{X_2^2+2X_3X_5}{6},\\
            b_7 &= \frac{X_0X_1}{2}, & b_8 &=\frac{X_0X_4+X_1X_2}{6}, & b_9 &=\frac{X_0X_3+X_2X_4}{6},\\
            b_{10} &= \frac{X_2X_3}{2}, & b_{11} &=X_1^2, & b_{12} &=\frac{X_1X_4}{2},\\
            b_{13} &= \frac{2X_1X_3+X_4^2}{6}, & b_{14} &=\frac{X_3X_4}{2}, & b_{15} &=X_3^2.
        \end{align*}
    \end{enumerate}
\end{examples}

\subsection{Waring decompositions on rational curves}\label{SectionWaringDecomps}
In this section, let $m=1$. In this setting, we refer to $\bm q$-Symmetric tensors as \emph{binary} $\bm q$-Symmetric tensors. In this case $V_{\bm q}$, parametrized by binary forms, is a rational curve and $\bm q$-Symmetric decompositions correspond to Waring decompositions on the rational curve $V_{\bm q}$.
The aim of this section is to deduce bounds on the minimal length and uniqueness properties, up to permutation and rescaling, of a $\bm q$-Symmetric decomposition of a binary $\bm q$-Symmetric tensor. Furthermore, the first item of the following proposition improves by one the bound in \cite[Thm. 4.1]{qiHankelTensors}.
Our approach has the benefit of highlighting the connections between $\bm q$-Symmetric decompositions of a binary $\bm q$-Symmetric tensor $p$ and Waring decompositions of the binary form $\psi_{\bm q}(p)$.

\begin{proposition}\label{RankBound}
Assume $W_{\bm q}$ is surjective. Let $p\in\bm q\text{-Sym}_{k,n+1}$.
\begin{enumerate}
    \item The minimal length $r$ of a $\bm q$-Symmetric decomposition over $\KK$ of $p$ is at most $hk$.
    \item If $p$ has a $\bm q$-Symmetric decomposition over $\KK$ of length $r\leq\lfloor \frac{hk+1}{2}\rfloor$, then the decomposition is \emph{unique} and $r$ is the minimal length of a $\bm q$-Symmetric decomposition of $p$ over $\KK$.
    \item Let $hk=2r-1$ and $\KK=\CC$. Then if $p$ is generic (i.e. $\psi_{\bm q}(p)$ is generic in $\S^{hk}(\CC^2)$), the minimal length of a $\bm q$-Symmetric decomposition of $p$ over $\CC$ is equal to $r$ and there exists \emph{a unique} length $r$ $\bm q$-Symmetric decomposition of $p$ over $\CC$.
    \item Let $hk=2r-2$ and $\KK=\CC$. Then if $p$ is generic (i.e. $\psi_{\bm q}(p)$ is generic in $\S^{hk}(\CC^2)$), the minimal length of a $\bm q$-Symmetric decomposition of $p$ over $\CC$ is equal to $r$ and there exist \emph{infinitely many} length $r$ $\bm q$-Symmetric decompositions of $p$ over $\CC$.
\end{enumerate}
\end{proposition}
\begin{proof}
We prove the statements separately.
    \begin{enumerate}
        \item Using \cref{QSymMainTheorem}, we can reformulate the statement as a property of $\psi_{\bm q}(p)$, thus we have to prove that \emph{$\psi_{\bm q}(p)$ has Waring rank $r$ over $\KK$ which is at most $hk$}.
        We give a short proof for $\KK=\CC$. For the general case, see \cite[Thm. 4.10]{reznickBinaryForms} and \cite[Prop. 2.1]{ComonOttavianiBinaryForms} for $\KK=\RR$.
        Let $\psi_{\bm q}(p)=\sum_{j=0}^{hk}\binom{hk}{j}c_jZ_0^{hk-j}Z_1^j$. After a change of variables, we can assume that either $c_{0}=c_{{hk}}=0$ or $c_{0}=c_{{hk}}=1$. In both cases, $g(Z_0,Z_1)=Z_1^{hk}-Z_0^{hk}=\prod_{l=0}^{hk-1}(Z_1-e^{\frac{2\pi il}{hk}}Z_0)\in \psi_{\bm q}(p)^\perp$. By \cref{ApolarityLemma}, $\psi_{\bm q}(p)$ has Waring rank over $\CC$ less than or equal to $hk$ \cite[Corollary 2.7]{huang2025waringproblemcomplexbinary}.
        \item Using \cref{QSymMainTheorem}, we can reformulate the statement as a property of $\psi_{\bm q}(p)$, thus we have to prove that \emph{if $\psi_{\bm q}(p)$ has a Waring decomposition over $\KK$ of length $r\leq\lfloor\frac{hk+1}{2}\rfloor$, then the Waring decomposition is unique and $r$ is the Waring rank over $\KK$ of $\psi_{\bm q}(p)$}.
        Suppose that $\psi_{\bm q}(p)$ has two Waring decompositions over $\KK$ of lengths $r$ and $s\leq r$
        \begin{equation*}
            \psi_{\bm q}(p)=\sum_{i=1}^r \lambda_i(\xi_{i,0}Z_0+\xi_{i,1}Z_1)^{hk};\quad\psi_{\bm q}(p)=\sum_{j=1}^s \mu_j(\rho_{j,0}Z_0+\rho_{j,1}Z_1)^{hk}
        \end{equation*}
        and note that each linear form corresponds to a unique point in the projective line $\PP^1(\KK)$. After reordering the summands, assume the first $t$ projective points coincide, for some $0\leq t\leq s$ and that the remaining forms correspond to disjoint sets of projective points: $\{[\xi_{t+1,0}:\xi_{t+1,1}],\ldots,[\xi_{r,0}:\xi_{r,1}]\}\cap\{[\rho_{t+1,0}:\rho_{t+1,1}],\ldots,[\rho_{s,0}:\rho_{s,1}]\}=\emptyset$. Hence, subtracting the second decomposition from the first
        \begin{equation*}
            \sum_{l=1}^t (\lambda_l-\mu_l)(\xi_{l,0}Z_0+\xi_{l,1}Z_1)^{hk}+\hspace{-0.1cm}\sum_{i=t+1}^r \lambda_i(\xi_{i,0}Z_0+\xi_{i,1}Z_1)^{hk}-\hspace{-0.1cm}\sum_{j=t+1}^s \mu_j(\rho_{j,0}Z_0+\rho_{j,1}Z_1)^{hk}=0.
        \end{equation*}
        Since $r+s-t\leq 2r\leq hk+1$, the terms $(\xi_{1,0}Z_0+\xi_{1,1}Z_1)^{hk},\ldots,(\xi_{r,0}Z_0+\xi_{r,1}Z_1)^{hk},\allowbreak(\rho_{t+1,0}Z_0+\rho_{t+1,1}Z_1)^{hk},\ldots,(\rho_{s,0}Z_0+\rho_{s,1}Z_1)^{hk}$ are linearly independent and this forces $r=s=t$ and the two decompositions to coincide \cite[Prop. 3.5]{huang2025waringproblemcomplexbinary}.
        \item Using \cref{QSymMainTheorem}, we can reformulate the statement as a property of $\psi_{\bm q}(p)$, thus we have to prove that \emph{if $\psi_{\bm q}(p)$ is a generic binary form of odd degree $hk=2r-1$, its Waring rank over $\CC$ is $r$ and there exists a unique length $r$ Waring decomposition over $\CC$}.
        This is a classical result due to Sylvester \cite{Sylvester_James_Joseph_rank}.
        \item Using \cref{QSymMainTheorem}, we can reformulate the statement as a property of $\psi_{\bm q}(p)$, thus we have to prove that \emph{if $\psi_{\bm q}(p)$ is a generic binary form of even degree $hk=2r-2$, its Waring rank over $\CC$ is $r$ and there exist infinitely many length $r$ Waring decompositions over $\CC$}.
        This is a result first established by Gundelfinger \cite{Gundelfinger1887}. An elementary proof can be found in \cite[Corollary 3.7]{huang2025waringproblemcomplexbinary}.
    \end{enumerate}
\end{proof}
\subsection{Quadrature formulae on rational curves}\label{SectionPositiveWaringDecomp}

In this subsection we set $m=1$, $\KK=\RR$ and study positive Waring decompositions on rational curves. As was anticipated in \cref{prop:waringIFFcoeffsmomentsequence}, positive Waring decompositions of a form $p$ are strictly connected to truncated moment problems and atomic representing measures of the normalized coefficient sequence of $p$. In turn, truncated moment problems have classically been studied in the context of approximation of integrals and computation of quadrature formulae. After briefly introducing quadrature formulae, we prove two preparatory lemmata and characterize, in \cref{Hamburger}, the positive binary $\bm q$-Symmetric tensors.
Our strategy is to first reduce the decomposition problem into a truncated moment problem, and then use the theory of flat extensions of moment matrices. Finally, we apply our results to the problem of computing quadrature formulae for measures supported on rational curves.
\begin{definition}
    A \emph{quadrature formula} of strength $k\in\NN$, for a given positive Borel measure $\mu$, is a finitely atomic measure $\nu=\sum_{i=1}^r\lambda_i\delta_{\bm\xi_i}$ with $\lambda_1,\cdots,\lambda_r>0,\,\bm\xi_1,\ldots,\bm\xi_r\in\RR^n$ such that the moments of $\mu$ and $\nu$ coincide up to degree $k$. The $\bm\xi_i$ and $r$ are called the \emph{nodes} and \emph{number of nodes}, respectively.
\end{definition}
\begin{lemma}\label{GenericChangeOfVars}
    Let $p=\sum_{j=0}^{2d}\binom{2d}{j}p_jZ_0^{2d-j}Z_1^j\in\S^{2d}(\RR^2)$ be a non-zero binary form and let $H=\operatorname{Hankel}(\bm p)$ be the Hankel matrix of $\bm p=(p_j)_{j=0}^{2d}$. Set $r=\rank(H)$. Then, after a generic change of variables $(Z_0,Z_1)\mapsto (U_0,U_1)$, writing $p=\sum_{j=0}^{2d}\binom{2d}{j}p'_jU_0^{2d-j}
    U_1^j$, the Hankel matrix $H'=\operatorname{Hankel}(\bm p')$ of $\bm p'=(p'_j)_{j=0}^{2d}$ has the property that its first $r$ columns are linearly independent.
\end{lemma}
\begin{proof}
    If $H$ has full rank, we are done and no change of variables is needed. Assume instead that $r<d+1$ and consider any change of variables $(Z_0,Z_1)\mapsto (U_0,U_1)$. We identify $H'$ with the quadratic form on $\S^{d}(\RR^2)$ having $H'$ as associated matrix with respect to the standard monomial basis $\{U_0^{d-i}U_1^i\mid i=0,\ldots,d\}$ in the variables $U_0,\,U_1$. Linear independence of the first $r$ columns of $H'$ is equivalent to the condition
    \begin{equation}\label{kernelConditionProp3.21}
        \ker(H')\cap\mathrm{span}\left\{U_0^{d},U_0^{d-1}U_1,\ldots,U_0^{d-r+1}U_1^{r-1}\right\}=\{0\}.
    \end{equation}
    By \cite[Thm. 1.44]{IK}, the apolar ideal $p^\perp$ of $p$ is generated by two polynomials $s$ and $t$ where $\deg(s)=l$, $\deg(t)=2d+2-l$ and $l\leq 2d+2-l$. Since $l\leq 2d+2-l$, we have $l\leq d+1$. If $l=d+1$, then also $2d+2-l=d+1$, so $p_d^\perp=\{0\}$, contradicting $r<d+1$ and $\ker(H)=p_d^\perp$ (cf. \cref{rem:catalecticantproperties}). Therefore, $l\leq d$, $2d+2-l>d$ and $t$ does not contribute to $p_d^\perp$. In particular, we deduce that $p_d^\perp$ is the degree $d$ homogeneous component of the ideal generated by $s$. Hence, $\ker(H)=\mathrm{span}\left\{s(Z_0,Z_1)\bm Z^{\bm \eta}\mid \bm \eta\in M_{d-l,2}\right\}$. By $\ker(H)=p_d^\perp$, we deduce that $r=\dim(\S^{d}(\RR^2))-\dim\ker(H)=(d+1)-(d+1-l)=l$, thus $\ker(H)=\mathrm{span}\left\{s(Z_0,Z_1)\bm Z^{\bm \eta}\mid \bm \eta\in M_{d-r,2}\right\}$.
    Applying the change of variables $(Z_0,Z_1)\mapsto (U_0,U_1)$, the kernel of $H$ is mapped to the kernel of $H'$, and we deduce
    \begin{equation}\label{kernelExpression2}
        \ker(H')=\mathrm{span}\left\{s(U_0,U_1)\bm U^{\bm \eta}\mid \bm \eta\in M_{d-r,2}\right\}.
    \end{equation}
    Given \cref{kernelExpression2}, any non-zero element in $\ker(H')$ must have $s$ as a factor. If $U_0$ is chosen so that it does not divide $s$, then the intersection in \cref{kernelConditionProp3.21} is trivial. In fact, since $s$ has degree $r$ and the forms in $\mathrm{span}\left\{U_0^{d},U_0^{d-1}U_1,\ldots,U_0^{d-r+1}U_1^{r-1}\right\}$ have at most degree $r-1$ in $U_1$, none of them can be a multiple of $s$.
\end{proof}

\begin{lemma}\label{PSDnessPreserved}
    Let $p(Z_0,Z_1)=\sum_{j=0}^{2d}\binom{2d}{j}p_jZ_0^{2d-j}Z_1^j\in\S^{2d}(\RR^2)$ and let $H=\operatorname{Hankel}(\bm p)$ be the Hankel matrix of $\bm p=(p_j)_{j=0}^{2d}$. Let $T$ be an invertible $2\times2$ real matrix, $(Z_0,Z_1)\mapsto T(U_0,U_1)$ be a change of variables and let $p(T(U_0,U_1))=\sum_{j=0}^{2d}\binom{2d}{j}p'_jU_0^{2d-j}U_1^j$. If $H'=\operatorname{Hankel}(\bm p')$ is the Hankel matrix of $\bm p'=(p'_j)_{j=0}^{2d}$, then $H'$ is PSD if and only if $H$ is PSD. Moreover, $\rank(H')=\rank(H)$.
\end{lemma}
\begin{proof}
    Let $L$ and $L'$ be the homogeneous moment functionals associated respectively with $\bm p$ and $\bm p'$, so that $H=M_d^{\mathrm h}(\bm p)$ and $H'=M_d^{\mathrm h}(\bm p')$ are the truncated moment matrices of $L$ and $L'$ respectively. Since $p(TU)$ has normalized coefficient sequence $\bm p'$, we have $L'(g(X))=L(g(T^\top X))$ for every homogeneous form $g$ of degree $2d$.
    Let $\mathcal B_X=\{X_0^d,X_0^{d-1}X_1,\ldots,X_1^d\}$ be the standard monomial basis of $\S^d(\RR^2)$, and let $A$ be the invertible matrix defined by $[f(T^\top X)]_{\mathcal B_X}=A[f]_{\mathcal B_X}$, where $[\cdot]_{\mathcal B_X}$ denotes the coefficient tuple of $\cdot$ with respect to the basis $\mathcal{B}_X$. Therefore, for every $f\in\S^d(\RR^2)$,
    \begin{equation*}
        [f]_{\mathcal B_X}^{\top}H'[f]_{\mathcal B_X}=L'(f^2)=L\bigl(f(T^\top X)^2\bigr)=[f]_{\mathcal B_X}^{\top}A^\top HA[f]_{\mathcal B_X}.
    \end{equation*}
    Hence $H'=A^\top HA$ and so $H$ and $H'$ have the same rank, and $H'\succeq0$ if and only if $H\succeq0$.
\end{proof}

The following result answers \cite[Question 4.2]{qiHankelTensors} positively and characterizes positive binary $\bm q$-Symmetric tensors.

\begin{proposition}\label{Hamburger}
    Let $m=1$, $k$ even. Suppose that $W_{\bm q}$ is surjective and let $p\in{\bm q}$-$\Sym_{k,n+1}$, $\psi_{\bm q}(p)=\sum_{j=0}^{hk}\binom{hk}{j}c_jZ_0^{hk-j}Z_1^j$ and $H=\operatorname{Hankel}(\bm c)$ be the Hankel matrix of $\bm c=(c_j)_{j=0}^{hk}$. Then, $p$ has a positive ${\bm q}$-Symmetric decomposition if and only if $H$ is PSD. Furthermore, the length of the positive ${\bm q}$-Symmetric decomposition is $r=\rank(H)$.
\end{proposition}
\begin{proof}
    Let $r=\rank(H)$ and assume that $r<hk/2+1$. Let $H_0$ be the principal $hk/2\times hk/2$ submatrix of $H$, which is indexed by all monomial multiples of $Z_0$. If $\rank(H_0)=r$, we can directly apply \cref{FlatExtensionForDecomps} and deduce that $p$ has a positive ${\bm q}$-Symmetric decomposition of length $r$. Else, if $\rank(H_0)<r$, we can perform a generic change of variables $(Z_0,Z_1)\mapsto(U_0,U_1)$ such that the resulting Hankel matrix $H'$ has rank $r$, has the principal $hk/2\times hk/2$ submatrix $H'_0$ of rank $r$ (see \cref{GenericChangeOfVars}) and is still PSD (see \cref{PSDnessPreserved}). Hence, we can apply \cref{FlatExtensionForDecomps} and deduce that $p$ has a positive ${\bm q}$-Symmetric decomposition of length $r$.
    If, instead, $r=hk/2+1$ we first show how to \emph{flat extend} the matrix $H$ to a larger matrix $\tilde{H}$, i.e. define $c_{hk+1}$ and $c_{hk+2}$ such that $\tilde{H}:=\operatorname{Hankel}(c_0,\ldots,c_{hk+2})$ and $\rank(\tilde{H})=\rank(H)$. Then, we apply \cite[Thm. 3.9]{curtofialkow1991} to complete the proof. Define $c_{hk+1}$ arbitrarily and, since $H$ is invertible, $c_{hk+2}$ as $c_{hk+2}=(c_{hk/2+1},\ldots,c_{hk+1})^\top H^{-1}(c_{hk/2+1},\ldots,c_{hk+1})$.
    This choice of $c_{hk+2}$ makes the Schur complement $c_{hk+2}-(c_{hk/2+1},\ldots,c_{hk+1})^\top H^{-1}(c_{hk/2+1},\ldots,c_{hk+1})$ of $H$ in $\tilde{H}$ vanish, so $\rank(\tilde{H})=\rank(H)$, and $\tilde{H}$ is PSD. Applying \cite[Thm. 3.9]{curtofialkow1991} to the extended tuple $(c_0,\ldots,c_{hk+2})$ yields that $(c_0,\ldots,c_{hk+2})$ is the truncated moment sequence of an $r$-atomic measure. As in the proof of \cref{prop:waringIFFcoeffsmomentsequence}, it is shown that $p$ has a positive ${\bm q}$-Symmetric decomposition of length $r$, completing the proof.
    The converse implication follows from \cref{FlatExtensionForDecomps}.
\end{proof}

\begin{theorem}\label{thm:auxiliary}
    Let $\varphi\in\RR[X]$ be a non-zero polynomial, $\bm c=(c_0,\ldots,c_{2N})$ be the truncated moment sequence of the positive Borel measure $\rho=\sum_{i=1}^{r'}\lambda_i\delta_{z_i}$, for ${r'}\in\NN,\,\lambda_1,\ldots,\lambda_{r'}>0$ and $z_1,\ldots,z_{r'}\in\RR$ distinct, $H=\operatorname{Hankel}(\bm c)$ and $r=\rank(H)\le N+1$. 
    If $z_1,\ldots,z_{r'}\in\RR\setminus\Z(\varphi)$, or $r=N+1$, then $\bm c$ has an atomic representing measure supported on $\RR\setminus\Z(\varphi)$ with at most $r$ atoms, which is unique when $r\le N$.
\end{theorem}
\begin{proof}
    We distinguish the two cases depending on whether $H$ is invertible or not. Let $L$ be the moment functional of $\bm c$ defined on $\RR[X]_{\leq 2N}$, i.e. $L(X^{i})=c_{i}$ for $i=0, \ldots, 2N$.
    
    Assume first that $H$ is not invertible, so $\rank(H)\le N$ and the kernel of $H$ is nontrivial. There exists a non-zero vector $\bm g=(g_0,\ldots,g_r, 0, \ldots, 0)\in\RR^{N+1}$. This implies that $\bm g^\top H\bm g=L(g(X)^2)=0$, where $g(X):=\sum_{i=0}^{r}g_i X^i$ is of degree $r$. Note that $0=L(g(X)^2)=\int g(X)^2\,\mathrm{d}\rho= \sum_{j=1}^{r'}\lambda_jg(z_j)^2$, so, since $\lambda_1,\ldots,\lambda_{r'}>0$, we deduce that $g(z_j)^2=0$ for all $j=1,\ldots,r'$. This implies that $\{z_1,\ldots,z_{r'}\}\subseteq\Z(g)$. In particular, $r'\leq|\Z(g)|\leq r$. Additionally, by \cite[Lemma 3.2.]{curtofialkow1991}, $r'\geq r$, so we deduce that $r'=r$ and uniqueness then follows from \cite[Thm. 3.10.]{curtofialkow1991}.
    
    Consider now the case $\rank(H)=N+1$. Our strategy is to extend the tuple $\bm c$ to $\tilde{\bm c}:=(c_0,\ldots,c_{2N},c_{2N+1})$ and rely on the theory of flat extensions of moment sequences (see \cite{curtofialkow1991}). To simplify the notation, denote $c_{2N+1}$ by $t$.
    We now show that a generic choice of $t$ yields an $N+1$-atomic representing measure $\nu$ for $\tilde{\bm c}$, and thus also for $\bm c$, supported on $\RR\setminus\Z(\varphi)$. Let $\bm v_i=(c_i,\ldots,c_{i+N})$ be the $i\textsuperscript{th}$ column of $H$. We construct a new column $\bm v_{N+1}$ which we define as a linear combination of the previous ones $\bm v_0,\ldots,\bm v_{N}$. Let $a_0,\ldots,a_{N}\in\RR$ be the coefficients of the linear combination, so that
    \begin{equation*}
    \begin{pNiceMatrix}[first-row]
    \bm v_0 & \bm v_1 & \cdots & \bm v_{N} \\
    c_0 & c_1 & \cdots & c_{N}\\
    c_1 & c_2 & \cdots & c_{N+1}\\
    \vdots & \vdots &  & \vdots\\
    c_{N} & c_{N+1} & \cdots & c_{2N}\\
    \end{pNiceMatrix}
    \begin{pNiceMatrix}[first-row]
    \bm a\\
    a_0 \\
    a_1 \\
    \vdots \\
    a_{N-1}\\
    a_{N}\\
    \end{pNiceMatrix}
    =
    \begin{pNiceMatrix}[first-row]
    \bm v_{N+1} \\
    c_{N+1} \\
    c_{N+2} \\
    \vdots \\
    c_{2N} \\
    t \\
    \end{pNiceMatrix}
    \end{equation*}

and, since $H$  is invertible, $\bm a=H^{-1}(c_{N+1},c_{N+2},\ldots,c_{2N},0)+tH^{-1}(0,0,\ldots,0,1)=\bm u + t \bm w$, where $\bm u:=H^{-1}(c_{N+1},\ldots,c_{2N},0)$ and $\bm w:=H^{-1}(0,\ldots,0,1)$. Since $H$ is the truncated moment matrix of order $N$ of a positive Borel measure, $H$ is PSD.
We consider the polynomial 
    \begin{align*}
        g_t(X)&=X^{N+1}-\left(\sum_{i=0}^N a_iX^i\right)
        =\left(X^{N+1}-\sum_{i=0}^N u_iX^i\right)-t\left(\sum_{i=0}^N w_iX^i\right)  =U(X)-tW(X),
    \end{align*}
also called \emph{generating polynomial} of $\tilde{\bm c}$. 
We show that, for a generic choice of $t$, $g_t$ and $\varphi$ have no commont real roots.
Let $z\in\Z(\varphi)$. Since $g_t(z)=0$ if and only if $U(z)-tW(z)=0$, we study separately the two following cases.

Case $W(z)=0$: In this case, for $z$ to be a root of $g_t$, $z$ must also be a root of $U(X)$. Let $U(X)=(X-z)\overline{U}(X)$ and $W(X)=(X-z)\overline{W}(X)$, where $\deg(\overline{U})\leq N$ and $\deg(\overline{W})\leq N-1$. Recall that $L$ is the moment functional of $\bm c$ defined on $\RR[X]_{\leq 2N}$, i.e. $L(X^{i+j})=c_{i+j}$ and evaluate $L$ at $(X-z)\overline{U}(X)\overline{W}(X)$. On the one hand, $L((X-z)\overline{U}(X)\overline{W}(X))=L(U(X)\overline{W}(X))$, write $\overline{W}(X)=\sum_{i=0}^{N-1}d_iX^i$ and note that
    \begin{align*}
        L(U(X)\overline{W}(X))&=\sum_{i=0}^{N-1}d_iL(U(X)X^i)
        =\sum_{i=0}^{N-1}d_iL\Bigg(X^{N+1+i}-\sum_{j=0}^N u_jX^{i+j}\Bigg)\\
        &=\sum_{i=0}^{N-1}d_i\Bigg(c_{N+1+i}-\sum_{j=0}^N u_jc_{i+j}\Bigg)
        =\sum_{i=0}^{N-1}d_i\left(c_{N+1+i}-c_{N+1+i}\right) =0,
    \end{align*}
    where the second to last equation follows from observing that $\bm u:=H^{-1}(c_{N+1},\ldots,c_{2N},0)$, so $H\bm u=(c_{N+1},\ldots,c_{2N},0)$, i.e. $\sum_{j=0}^{N}c_{i+j}u_j=c_{N+1+i}$ for all $i=0,\ldots,N-1$.
    
    On the other hand, since $U$ is a monic polynomial, so is $\overline{U}$, and hence it can be written as $\overline{U}(X)=X^N+S(X)$ for some polynomial $S$ of degree at most $N-1$. Additionally, since $L((X-z)\overline{U}(X)\overline{W}(X))=L(\overline{U}(X)W(X))$, we can compute $L(\overline{U}(X)W(X))=L(X^NW(X)+W(X)S(X))= L(X^NW(X))+L(W(X)S(X))=1+0=1$. Note that, in the last step, we used that $\bm w:=H^{-1}(0,\ldots,0,1)$, i.e. $H\bm w=(0,\ldots,0,1)$, yielding $L(X^NW(X))=(0,\ldots,0,1)^\top(0,\ldots,0,1)=1$, while $L(W(X)S(X))=0$ follows from $\deg(S)\leq N-1$, so the tuple of coefficients of $S$ in the $X$ variable is of the form $\bm s=(s_0,\ldots,s_{N-1},0)$, hence $L(W(X)S(X))=\bm s^\top H\bm w=\bm s^\top(0,\ldots,0,1)=0$. Computing $L((X-z)\overline{U}(X)\allowbreak\overline{W}(X))$ led to the contradiction $0=L(U(X)\overline{W}(X))=L(\overline{U}(X)W(X))=1$, hence we deduce that $U$ and $W$ cannot have a common root, so $W(z)=0$ implies that $g_t(z)\neq0$ and, in this case, $z\notin\supp(\nu)$.

    Case $W(z)\neq0$: In this case, $g_t(z)=0$ if and only if $t=U(z)/W(z)$.
    
    We conclude that, for $g_t$ and $\varphi$ to have distinct roots, it suffices that $t\notin\{\frac{U(z)}{W(z)}\mid z\in\Z(\varphi)\setminus\Z(W)\}$. In particular, a generic choice of $t$ meets the condition. By \cite[Cor. 3.4]{curtofialkow1991}, $\tilde{\bm c}$ has an $N+1$-atomic representing measure supported on the roots of the polynomial $g_t$, which are in $\RR \setminus \Z(\varphi)$.

    In summary, if $\rank(H)\leq N$, then $r'\leq r$ and $\rho$ is the desired measure. If, instead, $\rank(H)=N+1$, we are able to construct a representing measure for $\bm c$ supported on $\RR\setminus\Z(\varphi)$ with $N+1$ atoms by means of extending $\bm c$ to $\tilde{\bm c}$, and this completes the proof.
\end{proof}

The following result improves the bound \cite[Thm.1.1]{Riener2025QuadratureRW}, under an additional assumption on the parametrization of the algebraic curve $\mathcal{C}$.

\begin{theorem}\label{ComparisonWithRienerTuratti}
    Let $\bm q=(q_0,q_1,\ldots,q_n)$, with $q_i\in\S^h({\RR^2})$  with $W_{\bm q}$ surjective, $\bm \varphi:\RR\setminus \Z(q_0(1,Z))\to\RR^n;\;z\mapsto(q_1(1,z)/q_0(1,z),\ldots,q_n(1,z)/q_0(1,z))$ and assume that $q_0(1,Z)^k$ is non-negative on $\RR$. Let $\mathcal{C}:=\mathrm{Im}(\bm \varphi)$ and $\mu$ be a positive Borel measure supported on $\mathcal{C}$. 
    Let $p_{\bm\alpha}=\int X_1^{\alpha_1}\cdots X_n^{\alpha_n}\d\mu$ for $\bm\alpha=(\alpha_1, \ldots, \alpha_n)\in \NN^n,\, |\bm\alpha|\leq k$ and let
    \begin{equation*}
        p:=\sum_{\bm\alpha\in \NN_0^n,\, |\bm\alpha|\leq k}\binom{k}{(k-|\bm\alpha|,\alpha_1,\ldots,\alpha_n)}p_{\bm\alpha}X_0^{k-|\bm\alpha|}X_1^{\alpha_1}\cdots X_n^{\alpha_n}.
    \end{equation*}
    Then, $p$ is a $\bm q$-Symmetric tensor.
    Moreover, let $\psi_{\bm q}(p)=\sum_{j=0}^{hk}\binom{hk}{j}\allowbreak c_jZ_0^{hk-j}Z_1^j$, $H=\operatorname{Hankel}((c_0,\ldots,\allowbreak c_{hk}))$, $r=\rank(H)\leq \lceil\frac{hk+1}{2}\rceil$.
    Then, there exists a quadrature formula of strength $k$ for $\mu$ with 
    $r$ nodes on $\mathcal{C}$ if $r\neq \frac{hk+1}{2}$, and with $r+1$ nodes on $\mathcal{C}$ otherwise.
\end{theorem}
\begin{proof}
    Let $\varphi_i(Z):=q_i(1,Z)$ for $i=0,\ldots,n$ and let $\mu$ be a positive Borel measure supported on $\mathcal{C}$. Let $\bm p=\{p_{\bm\alpha}\mid\bm\alpha\in \NN_0^n,\, |\bm\alpha|\leq k\}$ denote the truncated moment sequence of $\mu$, i.e. $p_{\bm\alpha}=\int X_1^{\alpha_1}\cdots X_n^{\alpha_n}\d\mu$ for all $\bm\alpha\in \NN_0^n,\, |\bm\alpha|\leq k$. 
    Applying \cite[Thm. 2]{MR2231629} to the measure $\mu$, supported on $\mathcal{C}$, yields that $\bm p$ is also the truncated moment sequence of the atomic measure
    \begin{equation}\label{eq:atomicmeasure}
        \sum_{i=1}^{r'}\lambda_i\delta_{\bm\varphi(z_i)}
    \end{equation}
    where $\bm\varphi(z_i)\in\mathcal{C}$, for some distinct $z_1,\ldots,z_{r'}\in\RR\setminus\Z(\varphi_0)$ and $\lambda_1,\ldots,\lambda_{r'}>0$.
    Equivalently, defining
    \begin{equation*}
        p:=\sum_{\bm\alpha\in \NN_0^n,\, |\bm\alpha|\leq k}\binom{k}{(k-|\bm\alpha|,\alpha_1,\ldots,\alpha_n)}p_{\bm\alpha}X_0^{k-|\bm\alpha|}X_1^{\alpha_1}\cdots X_n^{\alpha_n},
    \end{equation*}
    by \cref{prop:waringIFFcoeffsmomentsequence}, we obtain that $p$ has the positive Waring decomposition
    \begin{equation}\label{PosWaringDec}
    p=\sum_{i=1}^{r'}\lambda_i\left(X_0+\frac{\varphi_1(z_i)}{\varphi_0(z_i)}X_1+\cdots +\frac{\varphi_n(z_i)}{\varphi_0(z_i)}X_{n}\right)^{k},
    \end{equation}
    where the linear forms correspond to the projective points $[1:\frac{\varphi_1(z_i)}{\varphi_0(z_i)}:\cdots:\frac{\varphi_n(z_i)}{\varphi_0(z_i)}]\in \mathrm{Im}(\bm q)$ for all $i=1,\ldots,r'$. We rewrite \cref{PosWaringDec} to deduce that
    \begin{equation}\label{AuxiliaryQSymDec}
        p=\sum_{i=1}^{r'}\frac{\lambda_i}{(q_0(1,z_i))^{k}}\left(q_0(1,z_i) X_0+\cdots +q_n(1,z_i)X_n\right)^{k}
    \end{equation}
    is a positive $\bm q$-Symmetric decomposition of $p$. Note that positivity of the decomposition is ensured by the non-negativity assumption on $q_0(1,Z)^k$. By \cref{QSymOnlyDec}, $p\in \bm q\text{-Sym}_{k,n+1}$ and so, applying \cref{QSymMainTheorem} to $p$, the binary form $\psi_{\bm q}(p)=\sum_{j=0}^{hk}\binom{hk}{j}\allowbreak c_jZ_0^{hk-j}Z_1^j$ has a positive Waring decomposition $\psi_{\bm q}(p)=\sum_{i=1}^{r'}\frac{\lambda_i}{(q_0(1,z_i))^{k}}\left(Z_0+z_iZ_1\right)^{hk}$. Equivalently, $c_j=\sum_{i=1}^{r'}\frac{\lambda_i}{(q_0(1,z_i))^{k}}z_i^j$ for all $j=0,\ldots,hk$, i.e. $\bm c=(c_j)_{j=0}^{hk}$ is the truncated moment sequence of a measure supported on $\RR\setminus \Z(\varphi_0)$.
   
    If $hk$ is even, in particular we have $r\neq\frac{hk+1}{2}$, and \cref{thm:auxiliary} applied to the sequence $\bm c$ yields an atomic representing measure supported on $\RR\setminus \Z(\varphi_0)$, with $r\leq hk/2+1$ atoms for $\bm c$.
    
    If $hk$ is odd and $r\neq\frac{hk+1}{2}$, arguing as in the first part of the proof of \cref{thm:auxiliary}, it is shown that $r'=r$ and $\bm c$ has the unique $r$-atomic representing measure \cref{eq:atomicmeasure}. If, instead, $r=\frac{hk+1}{2}$, let $\tilde{\bm c}=(c_{0}, \ldots, c_{hk}, c_{hk+1})$ where $c_{hk+1}=\sum_{i=1}^{r'}\frac{\lambda_i}{(q_0(1,z_i))^k}z_i^{hk+1}$. By \cref{thm:auxiliary} applied to $\tilde{\bm c}$, we obtain a representing measure for $\tilde{\bm c}$, so also for $\bm c$, supported on $\RR\setminus\Z(\varphi_0)$ with at most $r+1$ many atoms.
    Finally, we proceed in the reverse order as before: we deduce that $\psi_{\bm q}(p)$ has a positive Waring decomposition of length $r$, or at most $r+1$, respectively. Thus, $p$ has a positive $\bm q$-Symmetric decomposition of the same length, corresponding to a quadrature formula of strength $k$ for $\mu$ consisting of $r$ nodes, when $r\neq \frac{hk+1}{2}$, and at most $r+1$ nodes otherwise.
\end{proof}
\begin{remark}
    Under the additional non-negativity assumption on $q_0(1,Z_1)^k$ and surjectivity of $W_{\bm q}$, the bound we obtained in \cref{ComparisonWithRienerTuratti}
    is lower than \cite[Thm. 1.1 (2)]{Riener2025QuadratureRW} by the corank of the Hankel matrix $H$ and by the number of distinct real roots of $q_0(1,Z_1)$ in all cases, except when $hk$ is odd and the rank of $H$ is maximum, where our bound is lower by the number of distinct real roots of $q_0(1,Z_1)$ minus one.
\end{remark}
\begin{theorem}
    The bound obtained in \cref{ComparisonWithRienerTuratti} is sharp when $r\neq \frac{hk+1}{2}$. If $r= \frac{hk+1}{2}$, the bound is sharp whenever $q_0(1,Z_1)$ has a real root.
\end{theorem}
\begin{proof}
    We prove sharpness of the bound in \cref{ComparisonWithRienerTuratti} by demonstrating that, for any valid choice of $\bm q$ such that $W_{\bm q}$ is surjective, there exist positive Borel measures supported on $\C$ whose quadrature formulae strictly require the maximum number of nodes permitted by the bound. We analyze the two rank cases of $H$ separately, using the same notation as in \cref{ComparisonWithRienerTuratti}.

    Case $r\neq \frac{hk+1}{2}$: The bound in \cref{ComparisonWithRienerTuratti} states that a quadrature formula exists with at most $r$ nodes. We choose $r$ distinct arbitrary points $x_1, \dots, x_r \in \RR \setminus \Z(\varphi_0)$. We define a positive Borel measure $\nu$ on the real line by $\nu = \sum_{i=1}^{r} \omega_i \delta_{x_i}$, where $\omega_i > 0$ for all $i = 1, \dots, r$. Let $\bm c = (c_0, c_1, \dots, c_{hk})$ be the truncated moment sequence of $\nu$ up to degree ${hk}$, where each moment is defined by $c_j = \int_{\RR} X^j \, \d\nu = \sum_{i=1}^r \omega_i x_i^j$. We construct the corresponding measure $\mu= \sum_{i=1}^{r} \omega_iq_0(1, x_i)^k \delta_{\bm\varphi(x_i)}$ supported on the curve $\C$ via the rational parametrization $\bm\varphi$.
    Consider the Hankel matrix $H = \operatorname{Hankel}(\bm c)$ of size $r \times r$. Since $x_1,\ldots,x_r$ are distinct, the truncated moment matrix $H = \operatorname{Hankel}(\bm c)$ has rank $r$, every representing measure for $\bm c$ has at least $r$ atoms. Since any quadrature formula for $\mu$ of strength $k$ supported on $\C$ is the pushforward of a representing measure for $\bm c$, we deduce that no such quadrature formula with fewer than $r$ nodes can exist, proving sharpness of the bound.
    
    Case $r=\frac{hk+1}{2}$: Note that, in this case, $h$ and $k$ must be both odd. The bound in \cref{ComparisonWithRienerTuratti} yields at most $r +1$ nodes. We assume that $\Z(\varphi_0) \neq \emptyset$, so there exists a point $x^* \in \Z(\varphi_0)$ at which the parametrization $\bm\varphi$ is undefined. Let $x_1, \dots, x_{r-1} \in \RR \setminus \Z(\varphi_0)$ be distinct arbitrary points such that $x_i \neq x^*$ for all $i = 1, \dots, r-1$. We then define a positive $r$-atomic measure $\nu$ on the real line as $\nu = \omega^* \delta_{x^*} + \sum_{i=1}^{r-1} \omega_i \delta_{x_i}$, where $\omega^* > 0$ and $\omega_i > 0$ for all $i = 1, \dots, r-1$. Let $\bm c = (c_0, c_1, \dots, c_{hk})$ be the truncated moment sequence of $\nu$ up to the odd degree $hk$. The truncated moment matrix $H=\operatorname{Hankel}(\bm c)$ has maximum rank equal to $r$, so, according to \cite[Thm. 3.8 (ii)]{curtofialkow1991}, $\nu$ is the unique $r$-atomic representing measure for $\bm c$ and any other representing measure for $\bm c$ must consist of strictly more than $r$ atoms. In particular, there does not exist any representing measure for $\bm c$ with $r$ or fewer atoms supported on $\RR\setminus\Z(\varphi_0)$.
    
    We now show the existence of an $(r+1)$-atomic representing measure for $\bm c$ supported on $\RR\setminus\Z(\varphi_0)$. The truncated moment matrix $H'=\operatorname{Hankel}(c_0,\ldots,c_{hk-1})$ is PSD and also invertible, because $x^*,x_1,\ldots,x_{r-1}$ are distinct.
    We choose $c_{hk+1}\in\RR$ such that $c_{hk+1}>(c_{r},\ldots,c_{hk})^\top (H')^{-1}(c_{r},\ldots,c_{hk})$ and set $\tilde{\bm c}=(c_0,\ldots,c_{hk+1})$. By the theory of Schur complements, $\operatorname{Hankel}(\tilde{\bm c})$ is positive definite and, applying \cite[Thm. 3.8 (ii)]{curtofialkow1991} to $\tilde{\bm c}$, yields that $\tilde{\bm c}$ is the truncated moment sequence of a positive Borel measure. Furthermore, by \cite[Thm. 2]{MR2231629}, $\tilde{\bm c}$ is also the truncated moment sequence of a finitely atomic positive measure, so we can apply \cref{thm:auxiliary} to $\tilde{\bm c}$ and deduce the existence of an $(r+1)$-atomic representing measure $\sum_{i=1}^{r+1}\lambda_i \delta_{z_i}$ for $\tilde{\bm c}$, thus also for $\bm c$, supported on $\RR\setminus\Z(\varphi_0)$. 
    
    The measure $\mu=\sum_{i=1}^{r+1}\lambda_iq_0(1,z_i)^k \delta_{\bm\varphi(z_i)}$ has a quadrature formula of strength $k$, with $r+1$ nodes  supported on $\C$. 
    Any quadrature formula for $\mu$ with $r$ nodes on $\C$ is the pushforward of an $r$-atomic representing measure for $\bm c$ supported on $\RR\setminus\Z(\varphi_0)$, which  as discussed above, cannot exist.
    This shows that the bounds in \cref{ComparisonWithRienerTuratti} are sharp.
\end{proof}

\section{Effective decompositions of \texorpdfstring{$\bm q$}{q}-Symmetric tensors}\label{Sec4ExplicitDecomp}
In this section, under the assumption that $W_{\bm q}$ is surjective, we present an algorithm to explicitly compute a ${\bm q}$-Symmetric decomposition over $\KK$ of a ${\bm q}$-Symmetric tensor $p$, provided that a Waring decomposition over $\KK$ of $\psi_{\bm q}(p)$ is available. Finally, in \cref{SectionExampleDecomp}, we illustrate the use of \cref{AlgQSym} through an explicit example, computing a $\bm q$-Symmetric decomposition over $\RR$ of a monomial $\bm q$-Symmetric tensor and in \cref{SectionNumericalExperimentations} we probe the reliability of our implementation of \cref{AlgQSym} with the Julia package \href{https://github.com/matteobechere/QSymDecomposition.jl}{QSymDecomposition.jl}, available at \cite{QSymDecomposition}.

\begin{algorithm}[H]
\caption{${\bm q}$-Symmetric decompositions over $\KK$ of ${\bm q}$-Symmetric tensors}\label{AlgQSym}
\begin{algorithmic}[1]
    \Require Parameters $m,\,n,\,k,\,h\in\NN$, ${\bm q}=(q_0(\bm Z),\ldots,q_n(\bm Z))$,  with $q_i\in\S^h(\KK^{m+1})$, and $W_{\bm q}$ surjective, $p\in{\bm q}$-$\Sym_{k,n+1}$
    \Ensure A ${\bm q}$-Symmetric decomposition $p=\sum_{i=1}^r \lambda_i\langle \bm q(\bm\xi_i),\bm X\rangle^{k}$ over $\KK$
    \State Compute the matrix $\W$.
    \State Let $\A_{hk}=\mathrm{diag}((\binom{hk}{\bm\gamma}^{-1})_{\bm\gamma\in M_{hk,m+1}})$, $\A_{k}=\mathrm{diag}((\binom{k}{\bm\alpha}^{-1})_{\bm\alpha\in M_{k,n+1}})$ be the diagonal matrices with diagonals consisting of reciprocals of multinomial coefficients, ordered lexicographically. Compute $M=\A_k^{-1}\W^\top \A_{hk}$.
    \State Let $v$ be the coordinates of $p$ in the standard monomial basis of $\S^k(\KK^{n+1})$. Solve $Mw=v$ for $w$
    \State Compute $\psi_{\bm q}(p)$ as the form having $w$ as tuple of  coordinates in the standard monomial basis of $\S^{hk}(\KK^{m+1})$
    \State Compute a Waring decomposition over $\KK$ of $\psi_{\bm q}(p)$ and retrieve a decomposition $\psi_{\bm q}(p)=\sum_{i=1}^r\lambda_i\langle \bm\xi_i,\bm Z\rangle^{hk}$
    \State Output the ${\bm q}$-Symmetric decomposition $p=\sum_{i=1}^r \lambda_i\langle \bm q(\bm\xi_i),\bm X\rangle^{k}$
    
\end{algorithmic}
\end{algorithm}

\subsection{Explicit decomposition of a monomial \texorpdfstring{$\bm q$}{q}-Symmetric tensor}\label{SectionExampleDecomp}
In this subsection we explicitly compute a $\bm q$-Symmetric decomposition of a monomial $\bm q$-Symmetric tensor, using the following choice of ${\bm q}$
\begin{equation}\label{HighRankExample}
    \begin{gathered}
        q_0 = Z_0^2;\quad q_1 = Z_0 Z_1;\quad q_2 = Z_0 Z_2;\quad q_3 = Z_1^2;\quad
        q_4 = Z_1 Z_2;\quad q_5 = Z_2^2.
    \end{gathered}
\end{equation}
and $k=4$. This example showcases the following phenomenon: given a $\bm q$-Symmetric tensor $p$ whose Waring rank is too large for current methods to successfully decompose, by exploiting the $\bm q$-Symmetric structure, we are able to obtain a Waring decomposition by first computing $\psi_{\bm q}(p)$, performing a Waring decomposition of said form and finally mapping it back to a $\bm q$-Symmetric decomposition (in particular, a Waring decomposition) of $p$ via \cref{QSymMainTheorem}. By \cref{SurjectivityConditions}, the choice \cref{HighRankExample} guarantees surjectivity of $W_{\bm q}$, hence \cref{QSymMainTheorem} applies. Consider the $\bm q$-Symmetric tensor
\begin{align*}
p &= (4X_0 - 2X_1 + 2X_2 + X_3 - X_4 + X_5)^4- (X_0 + 2X_1 - X_2 + 4X_3 - 2X_4 + X_5)^4\\
&+ (X_0 - X_1 + 2X_2 + X_3 - 2X_4 + 4X_5)^4+ (4X_0 + 2X_1 - 2X_2 + X_3 - X_4 + X_5)^4\\
&- (X_0 + X_1 + X_2 + X_3 + X_4 + X_5)^4
+ (X_0 - 2X_1 - X_2 + 4X_3 + 2X_4 + X_5)^4\\
&+ (4X_0 + 2X_1 + 2X_2 + X_3 + X_4 + X_5)^4.
\end{align*}
Expanding the previous expression, we obtain a form consisting of $126$ monomials. For the sake of brevity, we only report the first and the last $3$, with respect to the lexicographic ordering.
\begin{equation*}
    p=768X_0^4 + 488X_0^3X_1 + 516X_0^3X_2 + \cdots  + 396X_4^2X_5^2 - 504X_4X_5^3 + 258X_5^4.
\end{equation*}
By construction, $p$ has Waring rank at most $7$. Since the Waring rank is not small enough, the \verb|decompose| function of the \href{https://github.com/AlgebraicGeometricModeling/TensorDec.jl}{TensorDec.jl} Julia package \cite{TensorDec}, which is based on the symmetric tensor decomposition algorithm introduced in \cite{symmetrictensordecomp}, is not able to successfully compute a Waring decomposition. Nonetheless, by implementing \cref{AlgQSym} in the \verb|qsym_decompose| function of the \href{https://github.com/matteobechere/QSymDecomposition.jl}{QSymDecomposition.jl} Julia package \cite{QSymDecomposition}, we are able to retrieve the following $\bm q$-Symmetric decomposition
\begin{equation*}
\begin{adjustbox}{width=\linewidth, keepaspectratio}
{\footnotesize $\begin{aligned}
    p_{\mathrm{dec}}&=1296(0.166667X_0 - 0.166667X_1 + 0.333333X_2 + 0.166667X_3 - 0.333333X_4 + 0.666667X_5)^4\\
    &- 81(0.333333X_0 + 0.333333X_1 + 0.333333X_2 + 0.333333X_3 + 0.333333X_4 + 0.333333X_5)^4\\
    &+ 1296(0.666667X_0 + 0.333333X_1 + 0.333333X_2 + 0.166667X_3 + 0.166667X_4 + 0.166667X_5)^4\\
    &+ 1296(0.666667X_0 - 0.333333X_1 + 0.333333X_2 + 0.166667X_3 - 0.166667X_4 + 0.166667X_5)^4\\
    &+ 1296(0.166667X_0 - 0.333333X_1 - 0.166667X_2 + 0.666667X_3 + 0.333333X_4 + 0.166667X_5)^4\\
    &+ 1296(0.666667X_0 + 0.333333X_1 - 0.333333X_2 + 0.166667X_3 - 0.166667X_4 + 0.166667X_5)^4\\
    &- 1296(0.166667X_0 + 0.333333X_1 - 0.166667X_2 + 0.666667X_3 - 0.333333X_4 + 0.166667X_5)^4
\end{aligned}$}
\end{adjustbox}
\end{equation*}
such that $\lVert p-p_{\mathrm{dec}}\rVert_4\sim\num{7.9377e-9}$, where $\lVert\cdot\rVert_4=\sqrt{\langle\cdot,\cdot\rangle_4}$ is the norm associated to the apolar product on $\S^4(\KK^6)$. To obtain $p_{\mathrm{dec}}$, we first computed $\psi_{\bm q}(p)$, obtaining
\begin{align*}
    \psi_{\bm q}(p) =&(2Z_0 - Z_1 + Z_2)^8- (Z_0 + 2Z_1 - Z_2)^8+ (Z_0 - Z_1 + 2Z_2)^8\\
&+ (2Z_0 + Z_1 - Z_2)^8 - (Z_0 + Z_1 + Z_2)^8 + (Z_0 - 2Z_1 - Z_2)^8 + (2Z_0 + Z_1 + Z_2)^8.
\end{align*}
When expanded, this form consists of $45$ monomials. Again, for the sake of brevity, we only report a few of them
\begin{equation*}
    \psi_{\bm q}(p)=768Z_0^8 + 976Z_0^7Z_1 + 1032Z_0^7Z_2 + \cdots + 1848Z_1^2Z_2^6 - 1008Z_1Z_2^7 + 258Z_2^8.
\end{equation*}
Furthermore, the Waring rank of $\psi_{\bm q}(p)$ is at most $7$, so in this case the \verb|decompose| function of the \href{https://github.com/AlgebraicGeometricModeling/TensorDec.jl}{TensorDec.jl} Julia package is able to successfully decompose $\psi_{\bm q}(p)$ and to compute
{\footnotesize \begin{align*}
    \psi_{\bm q}(p)_{\mathrm{dec}}&=1296(0.408248Z_0 - 0.408248Z_1 + 0.816497Z_2)^8 - 81(0.57735Z_0 + 0.57735Z_1 + 0.57735Z_2)^8\\
    &+ 1296(0.816497Z_0 + 0.408248Z_1 + 0.408248Z_2)^8
    + 1296(0.816497Z_0 - 0.408248Z_1 + 0.408248Z_2)^8\\
    &+ 1296(0.408248Z_0 - 0.816497Z_1 - 0.408248Z_2)^8
    + 1296(0.816497Z_0 + 0.408248Z_1 - 0.408248Z_2)^8\\
    &- 1296(0.408248Z_0 + 0.816497Z_1 - 0.408248Z_2)^8
\end{align*}
}%
such that $\lVert \psi_{\bm q}(p)-\psi_{\bm q}(p)_{\mathrm{dec}}\rVert_8\sim\num{1.411e-8}$.
Finally, we computed the decomposition $p_{\mathrm{dec}}$ of $p$ using \cref{QSymMainTheorem}.
\subsection{Quadrature on a rational curve}
In this section we present an example computation of a quadrature formula for a measure $\mu$ supported on a rational curve. In particular, we choose the denominator of the parametrization of the rational curve to have as many real zeros as possible, given the degree. Given the moments of the measure, we revisit the steps of the constructive proof of \cref{ComparisonWithRienerTuratti}: first of all, we define a $\bm q$-Symmetric tensor $p$ whose normalized coefficients are the moments of $\mu$. Then, we use \verb|qsym_decompose| on $p$ to obtain a $\bm q$-Symmetric decomposition. Finally, arguing as in the proof of \cref{ComparisonWithRienerTuratti}, we retrieve a quadrature formula for $\mu$.

In the notation of \cref{ComparisonWithRienerTuratti}, consider $k=4$ and
\begin{equation*}
    \bm q=((-Z_0+Z_1)(-2Z_0+Z_1)(3Z_0+Z_1),Z_0^2Z_1,Z_0Z_1^2,Z_1^3).
\end{equation*}
With this choice of $\bm q$, the rational curve has the parametrization given by $\bm\varphi:\RR\setminus\{-3,1,2\}\to\RR^3;\;t\mapsto (\frac{t}{(t-1)(t-2)(t+3)},\frac{t^2}{(t-1)(t-2)(t+3)},\frac{t^3}{(t-1)(t-2)(t+3)})$.
Since $\bm q$ is a basis of $\S^3(\RR^2)$, $W_{\bm q}$ is surjective (cf. \cref{CorollarySurjectivity}). Let $\mu$ be a positive Borel measure supported on $\C$ with truncated moment sequence up to degree $4$ given by
    
\begin{table}[H]
    \centering
    \small
    
\begin{tabular}{llll}
    \toprule
    
\multicolumn{4}{c}{Degree 0} \\
    \midrule
    $m_{(0,0,0)} = \num{4}$ & & & \\
    \midrule
    
\multicolumn{4}{c}{Degree 1} \\
    \midrule
    $m_{(1,0,0)} = \num{0.3390}$ & $m_{(0,1,0)} = \num{0.1398}$ & $m_{(0,0,1)} = \num{0.0612}$ & \\
    \midrule
    
\multicolumn{4}{c}{Degree 2} \\
    \midrule
    $m_{(2,0,0)} = \num{0.0478}$ & $m_{(1,1,0)} = \num{0.0216}$ & $m_{(1,0,1)} = \num{0.0101}$ & $m_{(0,2,0)} = \num{0.0101}$ \\
    $m_{(0,1,1)} = \num{0.0048}$ & $m_{(0,0,2)} = \num{0.0023}$ & & \\
    \midrule
    
\multicolumn{4}{c}{Degree 3} \\
    \midrule
    $m_{(3,0,0)} = \num{0.0078}$ & $m_{(2,1,0)} = \num{0.0037}$ & $m_{(2,0,1)} = \num{0.0018}$ & $m_{(1,2,0)} = \num{0.0018}$ \\
    $m_{(1,1,1)} = \num{8.3126e-4}$ & $m_{(1,0,2)} = \num{4.1416e-4}$ & $m_{(0,3,0)} = \num{8.3126e-4}$ & $m_{(0,2,1)} = \num{4.1416e-4}$ \\
    $m_{(0,1,2)} = \num{2.0660e-4}$ & $m_{(0,0,3)} = \num{1.0314e-4}$ & & \\
    \midrule
    
\multicolumn{4}{c}{Degree 4} \\
    \midrule
    $m_{(4,0,0)} = \num{0.0013}$ & $m_{(3,1,0)} = \num{0.0006}$ & $m_{(3,0,1)} = \num{0.0003}$ & $m_{(2,2,0)} = \num{0.0003}$ \\
    $m_{(2,1,1)} = \num{1e-4}$ & $m_{(2,0,2)} = \num{8.3126e-5}$ & $m_{(1,3,0)} = \num{1e-4}$ & $m_{(1,2,1)} = \num{8.3126e-5}$ \\
    $m_{(1,1,2)} = \num{4.1416e-5}$ & $m_{(1,0,3)} = \num{2.0660e-5}$ & $m_{(0,4,0)} = \num{8.3126e-5}$ & $m_{(0,3,1)} = \num{4.1416e-5}$ \\
    $m_{(0,2,2)} = \num{2.0660e-5}$ & $m_{(0,1,3)} = \num{1.0314e-5}$ & $m_{(0,0,4)} = \num{5.1520e-6}$ & \\
    \bottomrule
\end{tabular}
    
\end{table}

Let $p$ be the $\bm q$-Symmetric tensor obtained from the truncated moment sequence as in the proof of \cref{ComparisonWithRienerTuratti}, i.e. $p=m_{(0,0,0)}X_0^4+4m_{(1,0,0)}X_0^3X_1+4m_{(0,1,0)}X_0^3X_2+4m_{(0,0,1)}X_0^3X_3+\cdots+m_{(0,0,4)}X_3^4$.
Inputting $p$ in \verb|qsym_decompose| results in the following $\bm q$-Symmetric decomposition
{\footnotesize \begin{align*}
    p_{\mathrm{dec}}&=(X_0 + 0.190476X_1 + 0.095238X_2 + 0.047619X_3)^4+ (X_0 + 0.09X_1 + 0.03X_2 + 0.01X_3)^4\\
    &+ (X_0 + 0.058608X_1 + 0.014652X_2 + 0.003663X_3)^4
    + X_0^4,
\end{align*}}
such that $\lVert p-p_{\mathrm{dec}}\rVert_4\sim\num{3.806381e-11}$. We deduce that 
\begin{equation*}
    \delta_{(0.190476,0.095238,0.047619)}+ \delta_{(0.09,0.03,0.01)}+ \delta_{(0.058608,0.014652,0.003663)}+ \delta_{(0,0,0)}
\end{equation*}
is a quadrature formula of strength $4$ for $\mu$.
According to \cite[Thm.1.1 (2)]{Riener2025QuadratureRW}, given that $q_0(1,Z_1)$ has $3$ real zeros, $\mu$ has a quadrature formula of strength $4$ with at most $10$ nodes on $\C$. By \cref{ComparisonWithRienerTuratti}, we lower the upper bound to $7$.

\subsection{Numerical experimentation}\label{SectionNumericalExperimentations}
To assess the numerical reliability of the Julia implementation of \verb|qsym_decompose| for \cref{AlgQSym}, following the same procedure as in \cref{SectionExampleDecomp}, we developed the following experiment: for a fixed positive integer $r$ and a vector of forms $\bm q=(q_i)_i$ with $q_i\in\S^h(\RR^{m+1})$ such that the weight map $W_{\bm q}$ is surjective (see \cref{DefWeightFunction}), the code generates random vectors $\bm\lambda=(\lambda_1,\ldots,\lambda_r)\in\RR^r$ and random matrices $\Xi=(\xi_{i,j})_{0\leq i\leq m;1\leq j\leq r}\in\RR^{m+1\times r}$. The random real numbers are sampled from a standard normal distribution and the columns of $\Xi$ are normalized to be unit vectors with respect to the Euclidean norm. Furthermore, it computes the form $p=\sum_{j=1}^r \lambda_j(q_0(\xi_{0,j},\ldots,\xi_{m,j})X_0+\cdots +q_n(\xi_{0,j},\ldots,\xi_{m,j})X_n)^{k}\in\bm q\text{-Sym}_{k,n+1}$ and computes a $\bm q$-Symmetric decomposition of $p$ using \verb|qsym_decompose|. Finally, we measure the apolar norm of the difference of $p$ and the decomposition obtained with \verb|qsym_decompose|. We define a decomposition to be \emph{successful} if the apolar norm of the residual is smaller than $\num{1.e-4}$. The file \verb|NumericalReliabilityExperiment.jl|, available in the \verb|tests| folder at \cite{QSymDecomposition}, contains the source code to run this experiment.
The success rate is computed out of $50$ samples of $\bm\lambda$ and $\Xi$. The results for different choices of $\bm q$ and parameters $m,n,h,k$ are summarized for $m=2,n=8,h=3,k=3$ and $\bm q_1=(Z_0^3,Z_0^2Z_1,Z_0^2Z_2,Z_0Z_1^2,Z_0Z_2^2,Z_1^3,Z_1^2Z_2,Z_1Z_2^2,Z_2^3)$ in \cref{fig:success_rate_q1} and for $m=2,n=5,h=2,k=4$ and $\bm q_2=(Z_0^2,Z_0Z_1,Z_0Z_2,Z_1^2,Z_1Z_2,Z_2^2)$, as in \cref{HighRankExample}, in \cref{fig:success_rate_q2}. In the first example, the green curve corresponds to the performance of \verb|qsym_decompose| when the generating length of the $\bm q$-Symmetric tensor is specified as additional input, improving the accuracy. In the second example, the additional input did not yield any improvements in the performance.
\begin{figure}[tbhp]
\centering
\begin{tikzpicture}
\begin{axis} [
    width=0.85\linewidth,
    height=4cm,
    xlabel={$r$},
    ylabel={Success rate},
    ymin=0, ymax=1.05,
    xtick=data,
    ytick={0,0.25,0.5,0.75,1},
    grid=major,                 
    thick,
    legend cell align={left},   
    legend style={
    at={(0.02,0.05)},
    anchor=south west,
    legend columns=1,
    font=\small
},
    legend entries={
        \texttt{decompose},
        \texttt{qsym\_decompose} (length),
        \texttt{qsym\_decompose}
    }
]

\addplot[
    mark=*, color=red
]
coordinates {
(1,1.0) (2,1.0) (3,1.0) (4,1.0) (5,1.0)
(6,1.0) (7,1.0) (8,1.0) (9,1)
(10,0) (11,0) (12,0)
(13,0) (14,0) (15,0)
(16,0) (17,0) (18,0) (19,0)
};

\addplot[
    mark=square*, color=green!70!black
]
coordinates {
(1,1.0) (2,1.0) (3,1.0) (4,1.0) (5,1.0)
(6,1.0) (7,1.0) (8,1.0) (9,1.0)
(10,1.0) (11,1.0) (12,1.0)
(13,1) (14,1) (15,1)
(16,0) (17,0) (18,0) (19,0)
};

\addplot[
    mark=triangle*, color=blue
]
coordinates {
(1,1.0) (2,1.0) (3,1.0) (4,1.0) (5,1.0)
(6,1.0) (7,1.0) (8,1.0) (9,1.0)
(10,1.0) (11,1.0) (12,1.0)
(13,0.98) (14,0.9) (15,0.72)
(16,0) (17,0) (18,0) (19,0)
};

\end{axis}
\end{tikzpicture}
\caption{Success rate for $\bm q_1$ as a function of the rank $r$.}
\label{fig:success_rate_q1}
\end{figure}
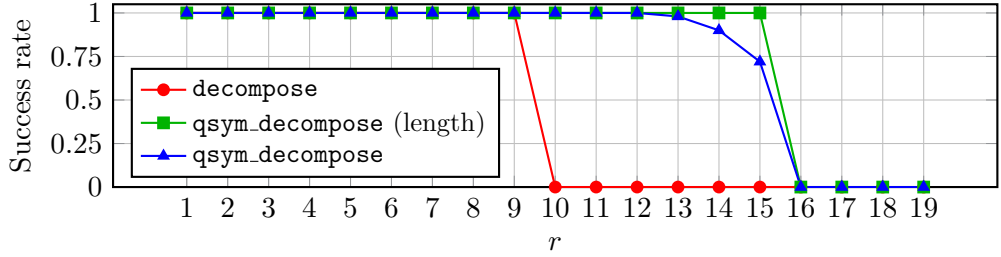

\begin{figure}[tbhp]
\centering
\begin{tikzpicture}
\begin{axis}[
    width=0.85\linewidth,
    height=4cm,
    xlabel={$r$},
    ylabel={Success rate},
    ymin=0, ymax=1.05,
    xtick={1,3,...,21},
    ytick={0,0.25,0.5,0.75,1},
    grid=major,                 
    thick,
    legend cell align={left},   
    legend style={
    at={(0.95,0.95)},
    anchor=north east,
    legend columns=1 ,
    font=\small
},
    legend entries={
        \texttt{decompose},
        \texttt{qsym\_decompose}
    }
]

\addplot[
    mark=*, color=red
]
coordinates {
(1,1.0) (2,1.0) (3,1.0) (4,1.0) (5,1.0)
(6,1.0) (7,0) (8,0) (9,0)
(10,0) (11,0) (12,0)
(13,0) (14,0) (15,0)
(16,0) (17,0) (18,0)
(19,0) (20,0) (21,0)
};

\addplot[
    mark=triangle*, color=blue
]
coordinates {
(1,1.0) (2,1.0) (3,1.0) (4,1.0) (5,1.0)
(6,1.0) (7,1.0) (8,1.0) (9,1.0)
(10,1.0) (11,0) (12,0)
(13,0) (14,0) (15,0)
(16,0) (17,0) (18,0)
(19,0) (20,0) (21,0)
};

\end{axis}
\end{tikzpicture}
\caption{Success rate for $\bm q_2$ as a function of the rank $r$.}
\label{fig:success_rate_q2}
\end{figure}
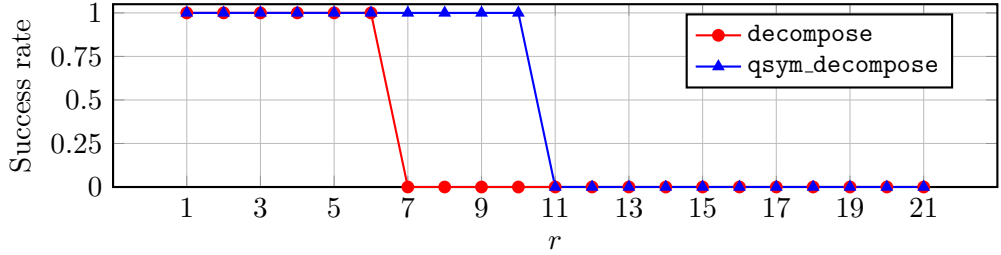

\section*{Acknowledgements}
The authors would like to thank Evelyne Hubert, Ettore Teixeira Turatti and Alja\v{z} Zalar for the helpful discussions and Henri Linus Breloer for the insightful comments leading to \cref{ExampleHenri}. Our work is supported by the European Union’s HORIZON-MSCA-2023-DN-JD programme under the Horizon Europe (HORIZON) Marie Sklodowska-Curie Actions, grant agreement 101120296 (TENORS).

\bibliographystyle{siamplain}
\bibliography{refs}

\end{document}